%% file: main.tex
\documentclass[letterpaper, 10 pt, conference]{IEEEtran}  
\IEEEoverridecommandlockouts 
\usepackage{cite}
\usepackage{amsmath,amssymb,amsfonts,mathptmx,colortbl}
 
\usepackage{amsthm,BOONDOX-cal}
\usepackage{graphicx}
\usepackage{enumerate}
\usepackage{relsize}
\usepackage{epsfig}
\usepackage[font=small,labelfont=bf]{caption}
\usepackage{bm}
\usepackage{epstopdf}
\usepackage{siunitx}
\usepackage{algorithm,array}
\usepackage[noend]{algpseudocode}
\usepackage{derivative}
\usepackage{multirow,multicol}

\usepackage{subfig,braket,enumitem,adjustbox,booktabs,slashbox}
\usepackage{color}
\usepackage{float}
\newfloat{algorithm}{t}{lop}
\usepackage{pgf,tikz}
\usetikzlibrary{arrows}
\usepackage{bbm}
\usepackage[overload]{empheq}
 \usepackage{tabularx, booktabs, makecell, caption,relsize}

\newcommand{\oprocendsymbol}{\hbox{$\bullet$}}
\newcommand{\oprocend}{\relax\ifmmode\else\unskip\hfill\fi\oprocendsymbol}

\newtheorem{theorem}{Theorem}

\newtheorem{lemma}[theorem]{Lemma}
\newtheorem{corollary}[theorem]{Corollary}
\newtheorem{remark}[theorem]{Remark}

\newcommand{\integernonneg}{\ensuremath{\mathbb{Z}_{\geq 0}}}
\newcommand{\real}{\ensuremath{\mathbb{R}}}
\newcommand{\complex}{\ensuremath{\mathbb{C}}}

\def\tablename{Table}

\newcommand{\Ec}{\mathcal{E}}

\newcommand{\Hc}{\mathcal{H}}

\newcommand{\Jc}{\mathcal{J}}

\newcommand{\Pc}{\mathcal{P}}
\newcommand{\Qc}{\mathcal{Q}}

\newcommand{\Tc}{\mathcal{T}}

\newcommand{\Vc}{\mathcal{V}}
\newcommand{\Wc}{\mathcal{W}}

\newcommand{\T}{^{\top}}

\newcommand{\Fb}{\overline{F}}

\newcommand{\dspi}{\displaystyle\int}

\newcommand{\until}[1]{\{1,\dots,#1\}}

\newcommand\pd{\partial}
\newcommand\eb{\overline{e}}
\newcommand\ebi{\overline{e}_i}
\newcommand\ebij{\overline{e}_{ij}}
\newcommand\ebj{\overline{e}_j}
\newcommand\ebk{\overline{e}_k}
\newcommand\srad{\text{srad}}

\newcommand{\cij}{\alpij}

\newcommand{\vqid}{\mu_q^i}

\newcommand{\Vi}{V^{-1}}
\newcommand{\Vit}{V^{-\top}}

\newcommand{\ub}{\overline{u}}

\newcommand{\longthmtitle}[1]{\mbox{}{\textit{(#1).}}}
\newcommand{\alpij}{\alpha_{ij}}
\newcommand{\VM}{\gamma}
\newcommand\td{\mathrm{d}}
\allowdisplaybreaks
\graphicspath{{epsfiles/}}
\usepackage{relsize}
\usepackage{tikz}
\tikzset{
>=stealth',
  punktchain/.style={rectangle, rounded corners, 
    draw=black, very thick,text width=10em, 
    minimum height=3em, text centered, on chain},
  line/.style={draw, thick, <-},
  element/.style={tape,top color=white,bottom color=blue!50!black!60!,
    minimum width=8em,draw=blue!40!black!90, very thick,
    text width=10em, minimum height=3.5em, text centered, on chain},
  every join/.style={->, thick,shorten >=1pt},
  tuborg/.style={decorate},
  tubnode/.style={midway, right=2pt},
}
\usepackage{eucal}
\usepackage{caption} 
\usepackage{tikz}
\usetikzlibrary{external}
\usepackage{standalone}
\usepackage{pgfplots}
\pgfplotsset{compat=1.6, scaled x ticks = false, xticklabel style={/pgf/number format/fixed,/pgf/number format/precision=3},legend image post style={xscale=0.3},
legend image post style={yscale=0.6}} 

\title{\LARGE \bf  Gramian-Informed Framework for Edge Flow Network Analysis Against Nodal Attacks\thanks{This work was supported in part by ARO Award W911NF-23-1-0138, SR/PURSE/2023/172 DST-PURSE Government of India, and by the National Laboratory of the Rockies (NLR) for the U.S. Department of Energy (DOE), operated under Contract No. DE-AC36-08GO28308. Funding for the NLR author was provided by the Laboratory Directed Research and Development (LDRD) Program at NLR. The views expressed do not necessarily represent those of the DOE or the U.S. Government. The U.S. Government retains, and the publisher, by accepting the article for publication, acknowledges that the U.S. Government retains a nonexclusive, paid-up, irrevocable, worldwide license to publish or reproduce this work for Government purposes.}}
\author{Prasad Vilas Chanekar,~\IEEEmembership{Member,~IEEE,} Bala Kameshwar Poolla,~\IEEEmembership{Member,~IEEE,} Jorge Cort{\'e}s,~\IEEEmembership{Fellow,~IEEE}%
\thanks{Prasad Vilas Chanekar is with the Department of Electronics and Communications Engineering, Indraprastha Institute of Information Technology, New Delhi, India 110020, {\tt prasad@iiitd.ac.in}, Bala Kameshwar Poolla is with the National Laboratory of the Rockies, Golden, CO 80401, USA, {\tt bpoolla@nlr.gov}, Jorge Cort{\'e}s is with the Department of 	Mechanical and Aerospace Engineering, University of California, San Diego, La Jolla, CA 92093, USA, {\tt cortes@ucsd.edu}}%
}

\begin{document}
\maketitle

\begin{abstract} 
We study the first-order effects of nodal inputs on edge flows of network systems through controllability Gramians. We consider both discrete- and continuous-time dynamics subject to impulse and step input disturbances. To characterize their impact on network behavior, we introduce the notion of a vulnerability matrix (VM) and provide explicit Gramian-based expressions that reveal the critical role played by the interplay between network topology and dynamics. For the class of directed line networks, we particularize these closed-form expressions in terms of edge weights, input duration, and graph distance to the nodal input. Numerical simulations on directed line and random Erd\H{o}s-R\'{e}nyi networks show a tight correspondence between influential nodes identified with the VM and conventional time-domain performance metrics, such as $\Hc_2$ and~$\Hc_{\infty}$.
\end{abstract}

\begin{IEEEkeywords}
network systems, controllability Gramian, edge-flow analysis, first-order effects, nodal vulnerability
\end{IEEEkeywords}

\section{Introduction}\label{introduction-1}
Complex networks are critical components of modern society, touching almost every aspect of daily life. We encounter them in social dynamics, intelligent transportation, water distribution systems, biological systems, energy systems, and robotics. Network evolution is naturally modeled as dynamics on graphs using nodes and edges. Inputs are applied at the nodes and get propagated across the network through flows on the edges, whose capacity is often constrained. The edge flows may represent various physical quantities depending upon the application, e.g., power (in energy networks), traffic (in transportation systems), or information (in communication networks). The ability to alter network behavior through nodal inputs and their propagation through capacity-constrained edges depends on the network topology as well as its dynamics. Malicious inputs at influential nodes can result in large, undesirable changes in edge flow characteristics and lead to network failures. Such influential nodes are ideal for adversarial attacks and make the network vulnerable. This paper develops formal tools to identify influential nodes by looking at the effect of impulse and step inputs on edge flows of networks modeled as linear time-invariant systems.

{\it Literature review:} Given a complex network~\cite{FB-LNS:22, FB-JC-SM:08cor} and a performance metric, the relative importance of a node or edge is quantified by the notion of centrality~\cite{MOJ:10, EE:12, MN:18}. Centrality is influenced by the network topology and the inherent dynamical properties of the nodes and edges; see~\cite{PVC-JC:22} for a list of various centrality measures. Traditionally, the study of centrality measures has revolved around the topology of the network, largely ignoring the interplay with its dynamical features. More recent work, see e.g.,~\cite {FP-SZ-FB:14, THS-FLC-JL:16, PVC-EN-JC:20} and references therein, studies the combined effect of network topology and the underlying dynamics. In \cite{GL-CA:20}, node-to-network and network-to-node influence are quantified using measures based on the controllability and observability Gramians. In \cite{SP-AD-MF-UV:14}, a Gramian-based analysis of a network's vulnerability to coordinated attacks on network links and data-injection attacks is presented.

{The work~\cite{IS-FF-THS:17} estimates the effect of attack inputs injected by an adversarial agent on the outputs and states of a network system using a novel security index based on the $\Hc_2-$norm. This norm is also employed in~\cite{MS-NM:13} to characterize robustness properties and fundamental limits of dynamical networks against external distributed stochastic disturbances and in~\cite{QH-YY-JG-MAD:14} to identify vulnerable nodes and interconnections against noisy nodal inputs. A different perspective is offered in~\cite{FP-SZ-CF_SZ:20}, which uses network fragility (computed using the stability radius) to characterize the network's response to external stimuli or parameter changes using network fragility, computed via the stability radius. Thus, in general, the vulnerability of a dynamical network system is analyzed using abstract system properties such as controllability, observability, or stability radius.

The literature on network flow vulnerability is comparatively limited. The work~\cite{AEM-YCL:2002} shows that when a node is attacked, the system load is redistributed to other nodes via edge flow redistribution, which may trigger cascading failures. Dynamic network flow vulnerabilities are analyzed in \cite{SS-DM-GZ:15} through a sequence of steady states. The work \cite{MT-LP-PJ:2019} argues that in synchronizable network-coupled dynamical systems such as high-voltage electric power grids, the synchrony is disturbed by nodal injections and presents a method based on spectral decomposition of the coupling matrix to identify the vulnerable nodes. An analytical study in~\cite{DM-MR-HR-XZ-SH-DW-MT:17} provides network susceptibilities quantifying the response of the dynamics to small parameter changes. Finally, \cite{BS-DW-MT-VL:18} highlights the failure of network components due to transient effects often overlooked by sequential steady-state analysis. More recently, our previous work \cite{PVC-BKP-JC:2023} demonstrated a  Gramian-based first-order relationship between changes in network flows arising due to impulsive nodal inputs. Building on this, \cite{MB-SME-AFT:2025} extends these results for shifted pulse and shifted step inputs and applies them to identify influential pipes in water networks.

\subsubsection*{Statement of Contributions} 
The contributions of this paper are threefold: (i) firstly, we provide mathematical expressions characterizing the first-order effects of nodal inputs, including both impulse and step inputs, on network edge flows for both discrete- and continuous-time dynamics. These expressions, derived in terms of the controllability Gramians of the underlying network, capture how localized perturbations propagate through the network topology. Further, we generalize the expressions for any matrix $A$, without requiring any stability assumptions. Next, based on these results, we propose the notion of (relative) Vulnerability Matrix (VM) to encode pairwise sensitivity relationships across all nodes and edges in a unified structure. Relative VM normalizes the flows to the system's steady-state flows; (ii) secondly, we derive detailed parametric expressions that capture vulnerability for the special class of directed line networks, revealing the explicit dependence on edge weights, path length between the input and observed nodes, and input duration; (iii) thirdly, we numerically validate the correspondence between the VM metrics and the observed time-domain evolution as computed from network dynamics. To this end, we introduce an algorithmic approach for determining nodal vulnerability at scale. This algorithm enables efficient computation of nodal vulnerability rankings for large networks, enabling practical vulnerability assessment for large networks. 

This work extends our preliminary, conference version in \cite{PVC-BKP-JC:2023} by: (a) relaxing the stability assumption on the network matrix $A$; (b) examining the first-order effects of step inputs on edge flows and numerically validating them; (c) deriving generic expressions for nodal vulnerability in terms of controllability Gramians and precise closed-form parametric expressions for a special class of directed-line networks-- without constraining the analysis to adjacent edges.}

\section{System Description}
Consider\footnote{\textit{Notation}: We denote by $\integernonneg$, $\real$, and $\complex$ the set of non-negative integers, real, and complex numbers, respectively. For $x\in \real$ (resp. $x\in \complex$), $\vert x\vert$ denotes its absolute value (resp. magnitude). The real part of $x\in \complex$ is denoted by $\mathcal{R}(x)$. The Dirac delta function is denoted by $\delta(t)$. The transpose of a vector or matrix is denoted by $(\cdot)^{\top}$. For $j \in \until{n}$, $\eb_j \in \real^n$ is the $j^\text{th}$ canonical unit vector (the bar is used to distinguish from the exponential~$e$). For a vector $x\in \real^n$, $x_i$ or $[x]_i$ denotes its $i^\text{th}$ coordinate. Given a  matrix~$M$, we denote its $(i,j)^\text{th}$ element by $m_{ij}$  or $M_{ij}$. For a square matrix $N$, we use $\lambda(N)$ to denote the vector of eigenvalues and $\srad(N)$ for its spectral radius.  We let $I$ denote the identity matrix of appropriate dimension.} a network consisting of $n$ nodes represented by the triplet $\mathcal{G}_A =
\left(\mathcal{V},\,\mathcal{E}_A,\,w_A\right)$, where
$\mathcal{V}=\lbrace 1,\,2,\ldots,n\rbrace$ is the node set, $\mathcal{E}_A=\set{\left(i,j\right) \mid i\in \mathcal{V}, j\in \mathcal{V}}$ is the edge set (with no self-loops) of cardinality $n_e$, and $w_A:\mathcal{E}_A\mapsto \real$ is the weight function. The pair $(i,j)$ denotes an edge directed from node $i$ to node $j$, i.e., $i\rightarrow j$. The weighted adjacency matrix $A \in\real^{n\times n}$ is defined by
$a_{ji} = w_A (i,j)\neq0$ if $\left(i,j\right)\in
\mathcal{E}_A$, else $a_{ji}=0$. The network dynamics are linear and either continuous- or discrete-time-invariant, with time evolution given by
\begin{subequations}\label{eq-dynamics-whole}
\begin{align}
\label{eq-dynamics-continuous}\dot{x} (t)&=Ax\left(t\right)+b_k\, \ub, \quad t\geq 0,\\
\label{eq-dynamics-discrete}x\left(t+1\right)&=Ax\left(t\right)+b_k\, \ub \quad t \in \lbrace0,\ldots,T-1\rbrace,
\end{align}
\end{subequations}
where  $x\in \real^n$, $t$ denotes time/iteration, and $T>0$ is a finite time horizon. Here $\ub\in \real$  is the input to the system applied at node $k$ at $t=0$, with its location denoted by the input vector $b_k=\eb_k$.  The input $\ub$ may be  an impulse $\ub = u\,\delta(t)$ or a step $\ub=u$ with magnitude $u$. The initial state is $x(0)=x_0 \in \real^n$. Note that we make no stability assumptions on the matrix $A$.

Let the $n_e$ edges in the edge set $\Ec_A$ be numbered consecutively and let $i\rightarrow j$ be the $l^{th}-$edge. We define the output for the dynamics~\eqref{eq-dynamics-whole} along the edge $i\rightarrow j$ at time $t$ as
\begin{align}
\label{eq-flow}
y_l(t)=F_{ij}(t) = \alpij(x_i(t) -x_j(t)),
\end{align}
where $\alpij \in\real$ is a known flow coefficient constant. We note that $y_l\in \real$ in \eqref{eq-flow} indirectly depends on the input $\ub$ through~$x$ due to the system dynamics \eqref{eq-dynamics-whole}. The output and variations thereof defined in~\eqref{eq-flow} are often used to quantify network edge flows~\cite{FD-MC-FB:13, DM-MR-HR-XZ-SH-DW-MT:17} and are relevant in network theory. In power systems, the matrix $A$ may be a Laplacian, with the choice $\alpij=a_{ji}$ representing line susceptances, and $x$ the vector of bus voltage angles. The quantity \eqref{eq-flow} then represents the active power flow on the line~\cite{FD-MC-FB:13} under the DC power flow approximation of $1~\mathrm{p.u.}$ voltage. For analysis, disturbances such as lightning strikes, electrostatic discharges, and utility fault clearing are often modeled as impulses. Thus, analyzing the effect of disturbances on edge flows is critical, as these are often capacity-constrained. Another perspective of interest relates to \eqref{eq-flow} being a weighted information flow in networks where the evolution of the states is governed by \eqref{eq-dynamics-whole}.

Let the vector $y=[y_1\; y_2\;\ldots\;y_{n_e}]\T$, an $n_e-$dimensional vector of edge flows of all the edges in the edge set $\Ec_A$ stacked vertically, represent the complete set of network edge flows. The performance measures~\cite{KZ-JD-KG:95} for analyzing systems include:
the average performance metric, given by the $\Hc_2$ norm,
\begin{subequations}
\label{performance-metrics-inf}
    \begin{align}\label{metric-h2}
         \Hc_2&:= \Vert y (t) \Vert_2^2 = \displaystyle\int_{0}^{\infty} y\T(t) \,y (t)\;\td t,    
    \end{align}
     and the worst-case performance metric, given by the $\Hc_{\infty}$ norm
    \begin{align}\label{metric-hi}
        \Hc_{\infty} &:= \sup_{\ub \neq 0}\;\frac{\Vert y(t)\Vert_2^2 }{\Vert \ub\Vert_2^2 }.
    \end{align}
\end{subequations}

Our objective is to determine the effect of exogenous inputs on the network edge flows and develop concepts and tools that quantify it, with the idea of facilitating the identification of network vulnerabilities. As we are dealing with a general class of systems, we consider finite-time performance metrics: the $\Jc_{2}$ performance modeled on the $\Hc_{2}$ norm and the maximum flow deviation performance $\Jc_{\infty}$ modeled on the $\Hc_{\infty}$ norm. Let $i\rightarrow j$ be the $q^{\text{th}}$ edge of the edge set $\Ec_A = \{1,\, 2,\ldots,q,\ldots,n_e\}$. As the output $y(t)$ is a vector of all edge flows stacked together, we have the $q^{\text{th}}-$ entry of $y(t)$, i.e., $y_q(t)=F_{ij}(t)$ and $y(t)\in\real^{n_e}$. 
Note that to evaluate the standard $\Hc_{2}$ and $\Hc_{\infty}$ metrics, we need to consider the edge flows over the entire time period $t\in[0, \infty)$. Thus, modeled on the above-mentioned performance measures, we define time-limited versions of these metrics, i.e., the time-limited average performance metric (analogously defined for discrete-time dynamics)
\begin{equation*}\label{performance-metrics}
         \Jc_2^f:= \Vert y (t) \Vert_2^2 = \displaystyle\int_{0}^{T} y\T(t) \,y (t)\;\td t,    
\end{equation*}
and the time-limited worst-case performance metric, i.e., the sum of the maximum absolute values of the change in flows of each edge for the simulation interval $[0,\; T]$, i.e.,
\begin{equation*}
\Jc^f_{\infty}:=\sum_{q=1}^{n_e}\, \underset{t\in[0,\, T]}{\max}\;\vert y_q(t)\vert=\sum_{q=1}^{n_e}\,\mathcal{F}_{ij},
\end{equation*}
where $\mathcal{F}_{ij}=\max|F_{ij}|$.

\section{Effect of Nodal Inputs on Edge Flows}\label{nodal-input}
In this section, we present a novel metric, termed the Vulnerability Matrix (VM), which quantitatively characterizes the effect of nodal inputs on the edge flows. The VM is determined by the network's topological properties as well as the inherent system-level dynamics. First, we derive analytical expressions for the evolution of edge flows for impulse nodal inputs/disturbances, followed by an analogous derivation for step inputs. We use the notation `$\;\vert^k\;$' to indicate the effect of an input exerted at node~$k$\footnote{When clear from context, we omit $\vert^k$ to denote the input node $k$ for improved readability.}.

\subsection{Nodal Impulse Inputs}
\label{impulse-input}
Here, we analyze the first-order effects of nodal impulse inputs on the network flows and compute the relevant analytical expressions. 
	
\begin{lemma}
\longthmtitle{Analytical expressions for evolution of flows \cite[Lemma 3]{PVC-BKP-JC:2023}}\label{lem-flow-expressions}
Consider the network dynamics \eqref{eq-dynamics-whole}. If an impulse input $\ub=u \,\delta(t)$ with $u\in\real$ is applied at node $k$ at $t=0$, then the time evolution of the flow $F_{ij}(t)$ on the edge $i\rightarrow j$, denoted by the output~\eqref{eq-flow} for an initial state $x(0)=x_0$, is given by
\begin{enumerate} 
\item for continuous-time dynamics,\\
$F_{ij}(t)\vert^k=\alpij(\eb_i-\eb_j)^{\top}e^{At}(x_0+\eb_k\,u), \; \text{for } t \geq 0,$
\item for discrete-time dynamics,\\
$ F_{ij}(t)\vert^k=\alpij(\eb_i-\eb_j)^{\top}A^{t-1}(A\,x_0+\eb_k\,u), \; \text{for } t \geq 1.$
\end{enumerate}
\end{lemma}

Next, we quantify the effect of impulse inputs on flows for continuous-time network dynamics through the controllability Gramian. {We define the finite-time controllability Gramian $\Pc^k_{\theta}$  for the time interval $[0,\;\theta]$ and input node $k$ as \cite{CTC:98} $$\Pc^k_{\theta}=\mathlarger{\sum}_{p=0}^{\theta-1}A^p\,\ebk\,\ebk\T\,{A^ p}\T$$ for the discrete-time dynamics and $$\Pc^k_{\theta}=\displaystyle\int_{0}^{\theta}e^{At}\,\ebk\,\ebk\T\, e^{A\T t}\,\td t$$ for the continuous-time case.}
    
\begin{theorem}\longthmtitle{First-order effects of impulse inputs on states and edge flows}
\label{thm-first-order-effects-impulse}	
Consider the continuous-time/discrete-time network dynamics in \eqref{eq-dynamics-whole}.  If an impulse input $\ub=u\,\delta(t)$ with $u\in\real$ is applied at node $k$ at $t=0$, then  we have the following for the node $i$ and the edge $i\rightarrow j$,
\begin{enumerate}[leftmargin=0.5cm,label=(\roman*)]
\item
$ \displaystyle\int_{0}^{{T}}\Big(\frac{\partial x_i}{\pd u}\Big)^2\bigg\vert^k\,\!\! \textup{d}t={\Qc}^k_{ii},$
\item $\displaystyle\int_{0}^{{T}}\Big(\frac{\pd F_{ij}}{\pd u}\Big)^2\bigg\vert^k\textup{d}t=\alpij^2\big({\Qc}^k_{ii}+{\Qc}^k_{jj}-2{\Qc}^k_{ij}\big),$
\item
$\mathlarger{\sum}_{t=0}^{{T}}\Big(\dfrac{\pd x_i}{\pd u}\Big)^2\bigg\vert^k={\Qc}^k_{ii},$
\item $\mathlarger{\sum}_{t=0}^{{T}}\Big(\dfrac{\pd F_{ij}}{\pd u}\Big)^2\bigg\vert^k=\alpij^2\big({\Qc}^k_{ii}+{\Qc}^k_{jj}-2{\Qc}^k_{ij}\big),$
	\end{enumerate}
where ${\Qc^k= \Pc^k_T}$.
\end{theorem}

\begin{proof}
For continuous-time dynamics \eqref{eq-dynamics-continuous} with an impulse input $\ub$, we have \cite{KO:10},
$$x_i(t)=\eb_i\T e^{At}(x_0+b_k\,u), \; \text{for } t \geq 0.$$	
Differentiating with respect to the input variable $u$ and noting that $b_k=\eb_k$ for the nodal input location, we have
\begin{align*}
		\frac{\pd x_i}{\pd u}=\eb_i^{\top}\,e^{At}\,\eb_k, \,\,\, \text{\it for } t \geq 0.
\end{align*}
To evaluate the overall impact over time, we integrate   from $t = 0$ to $\infty$, and obtain
$$\displaystyle\int_{0}^{{T}}\Big(\frac{\pd x_i}{\pd u}\Big)^2\,\textup{d}t =\displaystyle\int_{0}^{{T}} (\eb_i^{\top}\,e^{At}\,\eb_k)^2 \textup{d}t=\displaystyle\int_{0}^{{T}} (\eb_i^{\top}\,e^{At}\,\eb_k) (\eb_k\T e^{A\T t} \eb_i)\td t.$$ 
From the definition of the continuous-time controllability Gramian~\cite{KZ-JD-KG:95} for the input matrix $\eb_k$, we get 
$$\displaystyle\int_{0}^{{T}}\Big(\frac{\pd x_i}{\pd u}\Big)^2\,\textup{d}t = \eb_i\T {\Pc}^k_T \eb_i={\Qc}_{ii}^k,$$ establishing \emph{(i)}.
For \emph{(ii)}, we differentiate~\eqref{eq-flow} with respect to the input $u$ to get $$\dfrac{\pd F_{ij}}{\pd u}=\alpij\Big(\dfrac{\pd x_i}{\pd u}-\dfrac{\pd x_j}{\pd u}\Big).$$ On squaring and integrating with respect to $t$ from $0$ to $t\rightarrow \infty$, we obtain the expression
\begin{align*}
\displaystyle\int_{0}^{{T}}\Big(\frac{\pd F_{ij}}{\pd u}\Big)^2\td t=\alpij^2\displaystyle\int_{0}^{{T}}\Bigg(\Big(\frac{\pd x_i}{\pd u}\Big)^2+\Big(\frac{\pd x_j}{\pd u}\Big)^2-2\frac{\pd x_i}{\pd u}\;\frac{\pd x_j}{\pd u}\Bigg)\td t.
\end{align*} 	
Finally, we can compute the last term in the expression as 
\begin{align*}
\displaystyle&\int_{0}^{{T}}\frac{\pd x_i}{\pd u}\,\frac{\pd x_j}{\pd u}\td t=\displaystyle\int_{0}^{{T}}\eb_i\,e^{At}\,\eb_k\,\eb_j\,e^{At}\,\eb_k\td t,\\		&=\eb_i^{\top}\,\Big(\displaystyle\int_{0}^{{T}}e^{At}\,\eb_k\,\eb_k^{\top}\,e^{A\T t}\Big)\eb_j\,\td t=\,\,\,\eb_i^{\top}\,{\Pc}^k_T\,\eb_j={\Qc}^k_{ij}.
\end{align*}
The result for \emph{(ii)} now follows from \emph{(i)} and completes the proof for the continuous-time dynamics. We refer the reader to the conference version~\cite{PVC-BKP-JC:2023} for the proof of the discrete-time expressions in \emph{(iii)} and \emph{(iv)}.
\end{proof}

{In Theorem~\ref{thm-first-order-effects-impulse}, when $A$ is stable (i.e., $\srad(A) <1$  for discrete-time dynamics \eqref{eq-dynamics-discrete} or $\mathcal{R}\big(\lambda_i(A)\big)<0$, for all $i\in\{1,\,2,\ldots,n\}$, for continuous-time dynamics~\eqref{eq-dynamics-continuous}) and we take $T\rightarrow \infty$, we have $\Pc^k_{\infty}= \Wc^k$. Here, $\Wc^k$ is the infinite-time horizon controllability Gramian defined as }$\Wc^k=\mathlarger{\sum}_{p=0}^{\infty}A^p\,\ebk\,\ebk\T\,{A^ p}\T$ for discrete-time systems and $\Wc^k = \displaystyle\int_{0}^{\infty}e^{At}\,\ebk\,\ebk\T\, e^{A\T t}\,\td t$ for continuous-time systems. For a stable system matrix $A$, we can compute the controllability Gramian using the Lyapunov equations \cite{KZ-JD-KG:95} with input only at the node $k$ as the solution to
\begin{subequations}
\label{eq-Lyapunov}
\begin{align}
\label{eq-Lyapunov-continuous}\text{continuous-time:}\quad A\Wc^k+\Wc^k A^{\top}=-\eb_k\,\eb_k^{\top},
\\
\label{eq-Lyapunov-discrete}\text{discrete-time:}\quad A\Wc^k A^{\top}-\Wc^k=-\eb_k\,\eb_k^{\top}.
\end{align}
\end{subequations}

\begin{corollary}\longthmtitle{Relative first-order effect of impulse inputs on states and edge flows}\label{corollary-first-order-effects}
Consider the continuous-time/discrete-time network dynamics in \eqref{eq-dynamics-whole}. If an impulse input of magnitude $u\in\real$ is applied at node $k$ at $t=0$, then the relative effect of the impulse on the flow in the $i\longrightarrow j$ edge with a steady-state flow $\Fb_{ij}$ is		
\begin{enumerate}[leftmargin=0.5cm,label=(\roman*)]
\item $\dfrac{1}{\Fb^2_{ij}}\displaystyle\int_{0}^{{T}}\Big(\frac{\pd F_{ij}}{\pd u}\Big)^2\,\!\! \textup{d}t=\dfrac{\alpij^2}{\Fb^2_{ij}}\big({\Qc}^k_{ii}\!+\!{\Qc}^k_{jj}-2{\Qc}^k_{ij}\big)$,
\item $\mathlarger{\sum}_{t=0}^{{T}}\dfrac{1}{\Fb^2_{ij}}\Big(\dfrac{\pd F_{ij}}{\pd u}\Big)^2=\dfrac{\alpij^2}{\Fb^2_{ij}}\big({\Qc}^k_{ii}\!+\!{\Qc}^k_{jj}-2{\Qc}^k_{ij}\big)$,
\end{enumerate}
where ${\Qc^k=\Pc^k_T}$.
\end{corollary}

Next, we establish the equivalence between the $\Jc_2$ and $\Jc_{\infty}$ norms for first-order effects.
\begin{theorem}\longthmtitle{Equivalence of $\Jc_2$ and $\Jc_{\infty}$ first-order effects}\label{h2-hinf-1}
Consider the continuous-time/discrete-time network dynamics in \eqref{eq-dynamics-whole}. Let an impulse input of magnitude $u\in\real$ be applied at node $k$ at $t=0$; then the expressions for $\Jc_2$ and $\Jc_\infty$ for the performance output $y$	are identical when the system operates in the time interval $[0,\;T]$.
\end{theorem}
 
\begin{proof}    
Let the impulse input given at time $t=0$ at node $k$ be $\ub = u\; \eb_k \delta(t)$. With $x_0=0$, for the continuous-time dynamics, $$y_l =\alpij (\eb_i\T-\eb_j\T)x\,= \alpij \;u\;(\eb_i\T-\eb_j\T)\,e^{At}\,\eb_k,$$
where the edge $i\rightarrow j$ be the $l^{\text{th}}-$edge in the edge set $\Ec_A$ and $\alpij$ is the flow constant. Next, we compute the expressions for the $\Jc_2$ norm with respect to the output $y_l$ as
\begin{align*}
\Vert y_l\Vert_2^2 &= \dspi_0^{{T}}y_l^2(t)\; \td t\\=&\,\alpij^2\, u^2(\eb_i\T-\eb_j\T)\dspi_0^{{T}} e^{At}\,\eb_k\,\eb_k\T e^{A\T t}\td t \,(\eb_i-\eb_j),\\
& = \alpij^2\, u^2({\Qc}^k_{ii}+{\Qc}^k_{jj}-2{\Qc}^k_{ij}) =\alpij^2\, u^2(\eb_i\T-\eb_j\T){\Qc}^k \,(\eb_i-\eb_j),
\end{align*}

where ${\Qc^k=\Pc^k_T=\dspi_0^{{T}} e^{At}\,\eb_k\,\eb_k\T \,e^{A\T t}\td t}$ is the controllability Gramian of the system $(A,\;\eb_k)$. The expression $\Vert y_{l}\Vert_2^2$, in addition to denoting the system $\Jc_2-$ performance metric, also denotes the effect of the input of magnitude $u$, applied at node $k$ on the edge $i\rightarrow j$.

The $2$-norm of the input is computed as
$$\Vert \ub \Vert_2^2= \dspi_0^{{T}}(u \, \delta t)^2\,\eb_k\T\,\eb_k \; \td t=u^2.$$
On computing the ratio of the output to the input as below, 
$$\frac{\Vert y_l\Vert_2^2 }{\Vert \ub\Vert_2^2 }=\frac{\Vert y_l\Vert_2^2 }{u^2}= \alpij^2 ({\Qc}^k_{ii}+{\Qc}^k_{jj}-2{\Qc}^k_{ij}),$$
and evaluating the supremum on both sides, we get
\begin{align*}
\sup_{\ub \neq 0}\;\frac{\Vert y_l\Vert_2^2 }{\Vert \ub\Vert_2^2 }&=\sup\;\alpij^2 ( {\Qc}^k_{ii}+{\Qc}^k_{jj}-2{\Qc}^k_{ij}),\\
&= \alpij^2 ( {\Qc}^k_{ii}+{\Qc}^k_{jj}-2{\Qc}^k_{ij}).
\end{align*}
Note that $\sup_{\ub \neq 0}\;\frac{\Vert y_l\Vert_2^2 }{\Vert \ub\Vert_2^2 }$ is the $\Jc_{\infty}$ performance metric of the network system.

Similarly for the discrete-time dynamics from Lemma~\ref{lem-flow-expressions}, for $x(0)=0$, we have $y_l = \alpij \;u\;(\eb_i\T-\eb_j\T)A^{t-1}\,\eb_k$ and
\begin{align*}
\Vert y_l\Vert_2^2 &=\alpij^2\, u^2(\eb_i\T-\eb_j\T)\sum_{t=0}^{{T}} A^t\eb_k\,\eb_k\T\, {A^T}\T\,(\eb_i-\eb_j).
\end{align*}
The rest of the proof follows the same path as the continuous-time case. Thus $\Jc_{\infty}$ performance is equal to the $\Jc_2$ performance with $u=1$. Thus, the $\Jc_2$ and $\Jc_{\infty}$ behaviors of the network system are identical. 
 \end{proof}

\subsection{Nodal Step Inputs}
Beyond the impulse inputs considered above, many other real-world applications experience nodal inputs that can generally be modeled as step inputs of a certain duration. To cater to these systems, we characterize here the relative vulnerability to step nodal inputs. {Consider the operation of system \eqref{eq-dynamics-whole} in a time interval $[0,\;T]$. The system receives a step input for the interval $[0,\; T_s]$, with $T_s\leq T$.   Let the step interval $[0,\;T_s]$ be divided into $N$  equal parts such that {$t_0=0$}, $t_N=T_s$, and let $\tau=T_s/N$.} At each time instant $t_0,\,t_1,\ldots,t_N$, consider an impulse input. Thus, we model the step input as a train of impulse inputs and subsequently use this representation to characterize its effect on a network using the controllability Gramian. {Note that for the discrete-time case, $t_p = p$, which implies $N=T_s$.}

\begin{theorem}\longthmtitle{First-order effect of impulse-train input for discrete-time dynamics}\label{thm-first-order-effects-step-discrete}
Consider the discrete-time network dynamics in \eqref{eq-dynamics-whole}, and let a train of impulse inputs of magnitude $u\in\real$ be applied at node $k$ for {$t = 0,\,1,\ldots,N$} {with $N\geq 2$ and $N\leq T$}. Then,  the following hold true at node $i$ and edge $i\rightarrow j$,
    \begin{enumerate}[leftmargin=0.5cm, label=(\roman*)]
\item $\mathlarger{\sum}_{t=1}^{{T}}\Big(\dfrac{\pd x_i}{\pd u}\Big)^2=\ebi\T\mathcal{Q}^k\,\ebi,$
\item $ \mathlarger{\sum}_{t=1}^{{T}}\Big(\dfrac{\pd F_{ij}}{\pd u}\Big)^2= \cij^2\,\ebij\T\mathcal{Q}^k\,\ebij,$
	 \end{enumerate}
where $\ebij=(\ebi-\ebj),\, $
and
\begin{multline*}
 \mathcal{Q}^k =\Pc_1^k + \mathlarger{\sum}_{t=2}^{N}\mathlarger{\sum}_{p=1}^{t-1}\,\Big(\Pc_t^k\,+A^{t-p}\Pc_p^k\,+\,\Pc_p^k{A^{t-p}}\T\Big) 
 \\
 +\Big(\mathlarger{\sum}_{p={0}}^{N}A^p\Big)\,{\Pc^k_{T-N+1}}\,\Big(\mathlarger{\sum}_{p={0}}^{N}A^p\Big)\T .   
\end{multline*}
\end{theorem}

\begin{proof}
Consider the discrete-time dynamics in \eqref{eq-dynamics-whole}. We have $$\mathlarger{\sum}_{t=1}^{{T}}\Big(\dfrac{\pd x_i}{\pd u}\Big)^2=\mathlarger{\sum}_{t=1}^{N}\Big(\dfrac{\pd x_i}{\pd u}\Big)^2+\mathlarger{\sum}_{t=N+1}^{{T}}\Big(\dfrac{\pd x_i}{\pd u}\Big)^2$$ For $t\in \integernonneg$, with a train of impulse inputs of magnitude $u$ at node $k$ for {$t = 0,\,1,\ldots,N$}, the state evolution is given by~\cite{KO:10},
    $$x(t)=\begin{cases}
    A^t\,x_0+u\,\sum_{p=0}^{t-1} A^p\,\ebk,\quad t = 1,\,2,\ldots,N,\\
    A^t\,x_0+u\,\sum_{p={0}}^{N} A^{p+t-N-1}\,\ebk,\quad t = N+1,\,N+2,\ldots,{T}. 
    \end{cases}$$
Now for $t = 1,\ldots,N$ for any time $t$
    \begin{align*}
        \frac{\pd x_i}{\pd u}&=\ebi\T \frac{\pd x(t)}{\pd u}=\eb_i\T\sum_{p=0}^{t-1} A^p\ebk,\\
        \Big(\frac{\pd x_i}{\pd u}\Big)^2&=\eb_i\T\Big(\sum_{p=0}^{t-1}A^p\Big)\,\eb_k\,\eb_k\T\,\Big(\sum_{p=0}^{t-1}A^{p}\Big)\T\eb_i.
    \end{align*}
    By rewriting, $\mathlarger{\sum}_{p=0}^{t-1}A^{p}=I+\mathlarger{\sum}_{p=1}^{t-1}A^{p}$ we obtain
    \begin{align*}
        \Big(\frac{\pd x_i}{\pd u}\Big)^2&=\ebi\T\Big(\sum_{r=0}^{t-1} A^r\ebk\ebk\T{A^r}\T+ \sum_{p=1}^{t-1}\Big(\sum_{r=0}^{p-1}A^r\ebk\ebk\T{A^r}\T\Big)\,{A^{t-p}}\T+\\&\sum_{p=1}^{t-1}{A^{t-p}}\Big(\sum_{r=0}^{p-1}A^r\ebk\ebk\T{A^r}\T\Big)\,\Big)\,\ebi,\\
         &=\ebi\T\Big(\Pc_t^k+ \sum_{p=1}^{t-1}{A^{t-p}}\,\Pc_p^k\,+\sum_{p=1}^{t-1}\Pc_p^k\,{A^{t-p}}\T\Big)\,\ebi.
    \end{align*}
Next, for $t=N+1,\ldots,{T}$ we have $\frac{\pd x_i}{\pd u}=\eb_i\T\sum_{p={1}}^N A^{p+t-N-1}\ebk$, 
    $$\Big(\frac{\pd x_i}{\pd u}\Big)^2=\ebi\T\Big(\mathlarger{\sum}_{p={0}}^{N}A^p\Big)A^{t-N-1}\ebk\ebk\T{A^{t-N-1}}\T\Big(\mathlarger{\sum}_{p={0}}^{N}A^p\Big)\T\ebi.$$
\noindent Shifting $t$ by $N+1$ steps and taking the summation gives us
     \begin{align*}         
\mathlarger{\sum}_{t=N+1}^{{T}}\Big(\frac{\pd x_i}{\pd u}\Big)^2&=\ebi\T\Big(\mathlarger{\sum}_{p=0}^{N}A^p\Big)\mathlarger{\sum}_{t=0}^{{T-N}}A^t\ebk\ebk\T{A^t}\T\Big(\mathlarger{\sum}_{p=0}^{N}A^p\Big)\T\ebi,\\
&=\ebi\T\Big(\mathlarger{\sum}_{p={0}}^{N}A^p\Big){\Pc^k_{T-N+1}}\Big(\mathlarger{\sum}_{p={0}}^{N}A^p\Big)\T\ebi.
\end{align*}
Finally, computing the sum gives us the necessary result for {\emph(i)}. For computing the first-order effect on edge flows, we recall from Theorem \ref{thm-first-order-effects-impulse} that
    \begin{align*}\Big(\frac{\pd F_{ij}}{\pd u}\Big)^2=\alpij^2\Bigg(\Big(\frac{\pd x_i}{\pd u}\Big)^2+\Big(\frac{\pd x_j}{\pd u}\Big)^2-2\frac{\pd x_i}{\pd u}\;\frac{\pd x_j}{\pd u}\Bigg),
	\end{align*} 	
and from above, we have
    \begin{align*}
        \frac{\pd x_i}{\pd u}\,\frac{\pd x_j}{\pd u}=\begin{cases}
            \ebi\T\Big(\mathlarger{\sum}_{p=0}^{t-1}A^p\Big)\ebk\,\ebk\T\Big(\mathlarger{\sum}_{p=0}^{t-1}A^{p}\Big)\T\ebj,\; t = 1,\,2,\ldots,N,\\
            \ebi\T\Big(\mathlarger{\sum}_{p={1}}^{N}A^p\Big)A^{t-N-1}\ebk\ebk\T{A^{t-N-1}}\T\Big(\mathlarger{\sum}_{p={1}}^{N}A^p\Big)\T\ebj,\\
            \hspace{3.2cm}t = N+1,\,N+2,\ldots,{T}.
        \end{cases}
    \end{align*}
Proceeding as in \emph{(i)}, using $\ebij = \ebi-\ebj$, and taking the summation over $t=1$ to ${T}$ gives the required result for \emph{(ii)}.
    \end{proof}

Next, we present the results for continuous-time dynamics with a finite-time-step input, modeled as a train of impulses.

\begin{theorem}\longthmtitle{First-order effect of an impulse-train input for continuous-time dynamics}\label{thm-first-order-effects-step-continuous}
Consider the continuous-time network dynamics in \eqref{eq-dynamics-whole}. Let a train of impulse inputs of magnitude $u\in\real$ be applied at node $k$ at $t = t_0,\, t_1,\ldots{,t_N=T_s}$ such that $t_{q+1}-t_q = \tau$ $\forall$ $q = 0,\,1,\ldots,N-1.$  Then  we have the following for the node $i$ and the edge $i\rightarrow j$ with $t_0=0$,
    \begin{enumerate}[leftmargin=0.5cm, label=(\roman*)]
			\item 
			 $\dspi_{0}^{{T}}\bigg(\frac{\pd x_i}{\pd u}\bigg)^2 \textup{d}t=
       \ebi\T\,\mathcal{Q}^k\,\ebi,$
\item $ \displaystyle\int_{0}^{{T}}\Big(\frac{\pd F_{ij}}{\pd u}\Big)^2\,\textup{d}t=\cij^2\,\ebij\T\,\mathcal{Q}^k\,\ebij,$
	 \end{enumerate}
     where 
     $\ebij = \ebi-\ebj$,  $\Phi_q = \mathlarger{\sum}_{p=0}^{q}e^{pA\T\tau}$, and 
          $\mathcal{Q}^k=\mathlarger{\sum}_{q=0}^{N-1}\Phi_q\T\,\Pc^k_{\tau}\,\Phi_q+\Phi_N\T\,{\Pc^k_{T-T_s}}\,\Phi_N$.
      \end{theorem}
\begin{proof}
Consider the continuous-time dynamics in \eqref{eq-dynamics-whole}. We have $$\dspi_{0}^{{T}}\bigg(\frac{\pd x_i}{\pd u}\bigg)^2 \textup{d}t=\dspi_{0}^{t_N}\bigg(\frac{\pd x_i}{\pd u}\bigg)^2 \textup{d}t+\dspi_{t_N}^{{T}}\bigg(\frac{\pd x_i}{\pd u}\bigg)^2 \textup{d}t.$$
   \emph{(i)} The instantaneous state for a train of impulse inputs of magnitude $u$ applied at node $k$ with $t_0=0$ is  \cite{KO:10},
    \begin{align*}
        x(t) =\begin{cases}
        e^{At}x(0)+u\mathlarger{\sum}_{p=0}^{q}e^{A(t-t_p)}\ebk,\\
        \hspace{2cm} t_q\leq t < t_{q+1}, \quad q = 0,\,1,\ldots,N-1,\\
        e^{At}x(0)+u\mathlarger{\sum}_{p=0}^{N}e^{A(t-t_p)}\ebk, \quad t_N\leq t\leq T
        \end{cases} 
    \end{align*}
   Now for $q=1,\,2,\ldots,N-1$, we have
      \begin{align*}
        \frac{\pd x_i}{\pd u}&=\ebi\T \frac{\pd x(t)}{\pd u}=\eb_i\T \sum_{p=0}^{q}e^{A(t-t_p)}\ebk,\quad t_q\leq t < t_{q+1},\\
        \Big(\frac{\pd x_i}{\pd u}\Big)^2&=\eb_i\T\Big(\sum_{p=0}^{q}e^{A(t-t_p)}\Big)\eb_k\,\eb_k\T\Big(\sum_{p=0}^{q}e^{A(t-t_p)}\Big)\T\eb_i.
    \end{align*}
Integrating the overall effect from $t_0=0$ to $t_N$, we have
    \begin{align*}            
       &\dspi_{0}^{t_N}\bigg(\frac{\pd x_i}{\pd u}\bigg)^2 \textup{d}t\\        &=\sum_{q=0}^{N-1}\dspi_{t_q}^{t_{q+1}}\eb_i\T\bigg(\sum_{p=0}^{q}e^{A(t-t_p)}\bigg)\,\eb_k\,\eb_k\T\bigg(\sum_{p=0}^{q}e^{A(t-t_p)}\bigg)\T\eb_i\; \textup{d}t.
    \end{align*}    
    With substitutions $t=t_q+\theta,\,\tau = t_{q+1}-t_q,\, t_q-t_p=(q-p)\tau$,\\ $ \mathlarger{\sum}_{p=0}^{q}e^{A\T(t_q-t_p)}=\mathlarger{\sum}_{p=0}^{q}e^{pA\T\tau}=\Phi_q$ and changing limits of integration, we have
    \begin{align*}            
       &\dspi_{0}^{t_N}\bigg(\frac{\pd x_i}{\pd u}\bigg)^2\textup{d}t =\sum_{q=0}^{N-1}\ebi\T\Phi_q\T\dspi_{0}^{\tau}e^{A\theta}\,\eb_k\,\eb_k\T\,e^{A\T\theta} \textup{d}\theta\;\Phi_q\,\ebi.
    \end{align*}
Using  the definition of $\Pc^k_{\tau}$, we get 
$$\dspi_{0}^{t_N}\bigg(\frac{\pd x_i}{\pd u}\bigg)^2 \textup{d}t
     =\sum_{q=0}^{N-1}\ebi\T\Phi_q\T\,\Pc^k_{\tau}\,\Phi_q\,\ebi.$$
Next, for { $t_N\leq t \leq T$ } we have
\begin{align*}
\frac{\pd x_i}{\pd u}&=\eb_i\T \sum_{p=0}^{N}e^{A(t-t_p)}\,\ebk\,\, \textup{and}\\
\dspi_{t_N}^{{T}}\Big(\frac{\pd x_i}{\pd u}\Big)^2\!\td t\!&=\dspi_{t_N}^{{T}}\eb_i\T\Big(\sum_{p=0}^{N}e^{A(t-t_p)}\Big)\eb_k\,\eb_k\T\Big(\sum_{p=0}^{N}e^{A(t-t_p)}\Big)\T\eb_i\,\td t.
\end{align*}
Substituting $t = t_N+\theta$, changing limits of integration and with $t-t_p = t_N-t_p+\theta=(N-p)\tau+\theta$, we have
\begin{align*}
&\dspi_{t_N}^{{T}}\Big(\frac{\pd x_i}{\pd u}\Big)^2\!\td t=\eb_i\T\Big(\sum_{p=0}^{N}e^{pA\tau}\Big)\dspi_{0}^{{T}}e^{A\theta}\eb_k\,\eb_k\T e^{A\T \theta}\,\td\theta\Big(\sum_{p=0}^{N}e^{pA\T \tau}\Big)\eb_i.
\end{align*}
Using $t_N=T_s$, the definitions of $\Phi_N$ and $\Pc^k_{T-T_s}$, and summing the overall effect for $0\leq t\leq T$, we get the required result. 

\noindent\emph{(ii)} The result follows by proceeding as in the proof of Theorem \ref{thm-first-order-effects-step-discrete} and using \emph{(i)}.
\end{proof}

\begin{remark}\longthmtitle{Relative first-order effects for step inputs}\label{re:relative}
{\rm
Analogous to the results in Corollary~\ref{corollary-first-order-effects}, we can obtain expressions for relative first-order effects in response to step inputs (modeled as a train of impulses), with the absolute flow $F_{ij}$ relative to the steady-state flow $\Fb_{ij}$.} \oprocend
\end{remark} 

\subsection{Vulnerability Matrix}
We introduce here the notion of Vulnerability Matrix. Let $i\rightarrow j$ be the $l^\text{th}$ edge in the set $\Ec_A$ and let an input $u$ be given at node~$k$ of the network. Let the matrix $\VM\in \real^{n\times n_e}$ be defined for discrete-time/continuous-time dynamics as follows,
\begin{align}\label{eq:Vulnerability-Matrix}
    \sum_{t=0}^{{T}}\Big(\dfrac{\pd F_{ij}}{\pd u}\Big)^2=:\VM_{kl}\coloneqq\displaystyle\int_{0}^{{T}}\Big(\frac{\pd F_{ij}}{\pd u}\Big)^2\,\textup{d}t.
\end{align}
This matrix is referred to as the `\textbf{Vulnerability Matrix}' (VM) or the `Absolute Vulnerability Matrix'. The $k^\text{th}$ row of VM captures the effect of an input at node $k$ on the edge flow in each edge of the network. Instead, the $l^\text{th}$ column in VM represents the effect of an input (applied one at a time) at each node of the network on the $l^\text{th}$ network edge. The sum $\sum_{q=1}^{n_e}\VM_{kl}$ of the $k^\text{th}$ row elements encodes the influence of the $k^\text{th}$ input on all the edge flows. This is referred to as the \textbf{vulnerability influence} of node $k$ and denoted by~$\nu_k$. Higher values of $\nu_k$ make node $k$ a potential target for nodal attacks (viewed through an adversarial lens). An alternative notion (not explored here) is to define the vulnerability of a particular edge under input attacks at all nodes by summing the columns of VM, i.e., $\sum_{k=1}^{n}\VM_{k\,l}$.

A refinement of the notion of VM is the `{Relative Vulnerability Matrix}' or `{Normalized Vulnerability Matrix}', which relates to the intuition developed in Corollary~\ref{corollary-first-order-effects} and Remark~\ref{re:relative}. Here, the flows are normalized to the system's steady-state flows. This is often a critical metric, as absolute flows or changes in them without a baseline reference (e.g., steady-state flows) can lead to incorrect conclusions. While the definition~\eqref{eq:Vulnerability-Matrix} makes sense as a way of encoding network vulnerability, it offers little in terms of its actual evaluation. This is where the results presented above are consequential, since they offer an explicit, exact way to compute first-order effects without subjecting the system model to any external inputs. From \cite{FP-SZ-FB:14, THS-FLC-JL:16}, the controllability Gramian is related to the average controllability of the system when computed along all directions of the state space. This also corresponds to the energy contained in the output response to a unit impulse input \cite{KZ-JD-KG:95, PVC-JC:22}. Thus, Theorem~\ref{thm-first-order-effects-impulse} presents an energy-based expression in terms of the controllability Gramian, which quantifies the effect of impulse nodal inputs on edge flows. Similarly, Theorems~\ref{thm-first-order-effects-step-discrete} and \ref{thm-first-order-effects-step-continuous} quantify the effect of a train of impulse inputs on edge flows, expressed in terms of the controllability Gramians.
Furthermore, these expressions allow us to identify system vulnerabilities and strengths computationally. From a more conceptual perspective, this allows us to characterize these explicit expressions and their dependency on a variety of network properties. We tackle this in the next section for directed line networks.

\section{Analysis of directed line networks}
In the previous section, we derived closed-form expressions for analyzing the effect of nodal inputs on edge flows for general networks.\footnote{Here, we consider $T\to \infty$ and focus our attention on a special class of networks, i.e., directed line networks with a stable $A$ matrix (i.e., $\srad(A) <1$ or $\mathcal{R}\big(\lambda_i(A)\big)<0$, for all $\forall\,i\in\{1,\,2,\ldots,n\}$.}
This allows us to characterize the first-order effects in terms of network edge weights for nodal inputs. The treatment here generalizes to arbitrary edges in our conference paper~\cite{PVC-BKP-JC:2023}, which focused on the effect of impulse inputs on edges adjacent to the input node. Further, we also present results for step inputs.

 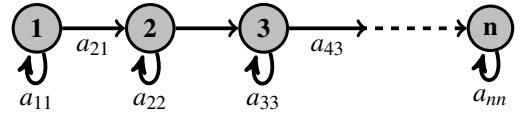
\begin{figure}[htb]
	\centering
		\vspace{-0.1in}
	\begin{tikzpicture}[line width=0.4mm]
		\node[shape=circle,draw=black, fill=lightgray, line width=0.4mm] (no1) at (0,0){$\bf 1$};
		\node[shape=circle,draw=black,fill=lightgray,line width=0.4mm] (no2) at (1.5,0){$\bf 2$};
		\node[shape=circle,draw=black,fill=lightgray, line width=0.4mm] (no3) at (3,0){$\bf 3$};
		\node[] (no4) at (4.5,0){};
		\node[] (no4a) at (4.125,0){};
		\node[shape=circle,draw=black,fill=lightgray,line width=0.4mm] (no5) at (6,0){$\bf n$};
		\draw[-to,line width = 0.5mm](no1)--node[below]{$a_{21}$}(no2);
		\draw[-to,line width = 0.5mm](no2)--(no3);
		\draw[-to,line width = 0.5mm](no3)--node[below]{$a_{43}$}(no4);
		\draw[dashed,-to,line width = 0.5mm](no4a)--(no5);
		\draw[->,line width = 0.5mm](no1) edge[loop below] node[below]{$a_{11}$}(no1);
		\draw[->,line width = 0.5mm](no2) edge[loop below] node[below]{$a_{22}$}(no2);
		\draw[->,line width = 0.5mm](no3) edge[loop below] node[below]{$a_{33}$}(no3);
		\draw[->,line width = 0.5mm](no5) edge[loop below] node[below]{$a_{nn}$}(no5);
	\end{tikzpicture}
	\caption{A directed line network with $n$ nodes.}\label{fig-directed-line}
\end{figure} 

Consider a directed line network, i.e., a sequential arrangement of nodes, starting at node $1$ and ending at node $n$ (see Fig.~\ref{fig-directed-line}). The presence of self-loops ensures the stability of the continuous-time network dynamics. We define the notation $\rho_{ij}$ as $\rho_{ii} = 1$ and for $i<j$, 
\begin{equation}
\label{note1}
\rho_{ij}\coloneqq a_{i+1,\,i}\,\,a_{i+2,\,i+1} \ldots a_{j,\,j-1} = \mathlarger{\prod}_{r=i}^{j-1}a_{r+1,r}.
\end{equation}
For $i\leq k\leq j$, we have $\rho_{ij}=\rho_{ik}\,\rho_{kj}$.

\subsection{Discrete-Time Directed Line Network}
Consider the discrete-time dynamics in \eqref{eq-dynamics-discrete}, with $A \in\real^{n\times n}$ representing the weighted adjacency matrix. Next, consider a special class of directed line networks as in Fig.~\ref{fig-directed-line} with no self-loops (i.e., $a_{ii}=0\,\, \forall\, \,i\in\{1,\ldots,n\}$) and positive edge weights (i.e., $a_{ji}>0$). For a directed line network, the infinite-time horizon controllability  Gramian $\Wc^k$ and finite-time controllability Gramian $\Pc^k_t$ with the input node at $k$ are diagonal in nature. We have 
 \begin{align}
     \label{IFT-gramian-discrete-line}
     \Wc^k_{rr} = \begin{cases}
          0, \quad &\text{for} \quad r<k,\\
          \rho^2_{kr}, \quad &\text{for} \quad k\leq r\leq n.
     \end{cases}     
 \end{align}
 The finite-time horizon controllability Gramian is as follows for $t\in\integernonneg,\;1\leq t\leq n-k+1$,
  \begin{align}
     \label{FT-gramian-discrete-line}
     [\Pc^k_t]_{rr} = \begin{cases}
         \rho^2_{kr}, \quad &\text{for} \quad k\leq r\leq k+t-1,\\
         0, \quad &\text{otherwise}.
     \end{cases}     
 \end{align}
 Note that for $t\geq n-k+1$ we have $\Pc^k_t=\Wc^k.$ 
 For any $r\in\integernonneg$, we examine the mathematical operation of multiplying $A^r$ and ${A^r}\T$ with the vector $\eb_s$. Note that $A^r\,\eb_s$ is a vector with exactly one nonzero element for $0\leq r \leq n-s$  and $A^r\,\eb_s = 0$ when $r>n-s$. The nonzero element is at the $(s+r)^{\text{th}}$ coordinate whose value is  $[A^r\,\eb_s]_{s+r}=\rho_{s,s+r}$ for $0\leq r\leq n-s$. Similarly ${A^r}\T\eb_s$ is a vector with exactly one nonzero element for $0\leq r \leq s-1$  and ${A^r}\T\eb_s = 0$ when $r\geq s$. The nonzero element is at the $(s-r)^{\text{th}}$ coordinate whose value is  $[{A^r}\T\eb_s]_{s-r}=\rho_{s-r,s}$ for $0\leq r\leq s-1$.  
 
\begin{theorem}\longthmtitle{Impulse-train input effect in directed line networks for discrete-time dynamics}\label{thm-line-network-step-discrete}
Consider the discrete-time line network dynamics in \eqref{eq-dynamics-whole}, and let a train of impulse inputs of magnitude $u\in\real$ be applied at node $k$ at $t = 0,\,1,\,2,\ldots,N$, with $N>2$. Then  we have the following for the node $i$ and the edge $i\rightarrow j$ with $j=i+1$,
    \begin{enumerate}[leftmargin=0.5cm, label=(\roman*)]
\item $\mathlarger{\sum}_{t=1}^{\infty}\Big(\dfrac{\pd x_i}{\pd u}\Big)^2= \begin{cases}
0, \qquad \text{for}\quad i<k,\\
\big(i+1+\max(N+k-i,0)\\-\max(k,i-N)\big)\,\rho_{ki}^2,\quad\text{for}\quad i\geq k.
\end{cases}$
\item $ \mathlarger{\sum}_{t=1}^{\infty}\Big(\dfrac{\pd \widetilde{F}_{ij}}{\pd u}\Big)^2=\begin{cases}
    0, \quad &\text{for}\quad i<k-1,\\
     N+1, \quad &\text{for}\quad i=k-1,\\   
     \Tc_1+\Tc_2-2\Tc_3-2\Tc_4 &\text{for}\quad i>k.
\end{cases} $
	 \end{enumerate}
     where $\widetilde{F}_{ij}=F_{ij}/\alpij$,\\ $\Tc_1=\big(i+1+\max(N+k-i,0)-\max(k,i-N)\big)\,\rho_{ki}^2,$\\
     $\Tc_2=\big(i+2+\max(N+k-i-1,0)-\max(k,i+1-N)\big)\,\rho_{k,i+1}^2,$\\
     $\Tc_3=\begin{cases}
         0, \quad &\text{for} \quad N<i-k+2,\\
         (N+k-i-1)\,\rho^2_{ki}\,\rho_{i,i+1}\quad &\text{for} \quad N\geq i-k+2,
     \end{cases}$\\
     $\Tc_4=(i+1-\max(i-N+1,k))\rho^2_{ki}\,\rho_{i,i+1}.$
\end{theorem}
    
\begin{proof}
We will treat each case separately. For {\emph{(i)}}, we use the results derived in Theorem \ref{thm-first-order-effects-step-discrete}.
\begin{enumerate}[leftmargin=0.5cm, wide, labelwidth=!, labelindent=0pt, label=(\Alph*)]
\item The first term of ${\sum}_{t=1}^{\infty}\Big(\tfrac{\pd x_i}{\pd u}\Big)^2$ equals $\ebi\T\sum_{t=1}^{N}\Pc^k_t\;\ebi$ i.e., the $i^{\text{th}}$ diagonal element of $\sum_{t=1}^{N}\Pc^k_t$. Now  $\ebi\T\sum_{t=1}^{N}\Pc^k_t\;\ebi=0$ for $N\leq i-k$ or $i<k$. Whereas, for $i\geq k$ and $N\geq i-k+1$, we have $\ebi\T\sum_{t=1}^{N}\,\Pc^k_t\;\ebi=\max(N+k-i,0)\,\rho^2_{ki}$.

\item The second term is $\ebi\T{\sum}_{t=2}^{N}{\sum}_{p=1}^{t-1} (A^{t-p}\Pc_p^k+\Pc_p^k{A^{t-p}}\T)\;\ebi$, furthermore, $\ebi\T A^{t-p}\,\Pc_p^k\;\ebi = \ebi\T\,\Pc_p^k{A^{t-p}}\T\;\ebi$. Note that $\ebi\T A^{t-p}=0$ for $t-p\geq i$. For $1\leq t-p\leq i-1$, $\ebi\T A^{t-p}$ is a row vector with only one nonzero element at the $(i-t+p)^{\text{th}}$ position. Meanwhile, $\Pc^k_p\,\ebi =0$ for $1\leq i < k$ or $1\leq p \leq i-k$. For $p\geq i-k+1 $ and $i\geq k$, $\Pc^k_p\,\ebi$   is vector with one nonzero element at the $i^{\text{th}}$ position. Since, $t-p\geq 1$, we have $i-t+p \neq i$, which implies the sum is always $\ebi\T{\sum}_{t=2}^{N}{\sum}_{p=1}^{t-1}\,(A^{t-p}\Pc_p^k+\Pc_p^k{A^{t-p}}\T)\;\ebi=0$.

\item Consider the third term $\ebi\T \Big({\sum}_{p={0}}^{N}A^p\Big)\,\Wc^k\,\Big({\sum}_{p={0}}^{N}A^p\Big)\T\, \ebi$. Since ${A^p}\T\ebi = 0$ for $p\geq i$, the nonzero elements of the vector $[({\sum}_{p={0}}^{N}A^p)\T\, \ebi]_s = \rho_{si}$ with $\max(i-N,1)\leq s\leq i$. 
Observe that, $\Wc^k({\sum}_{p={0}}^{N}A^p)\T\, \ebi=0$ for $i<k$. For $i\geq k$ we have $[\Wc^k({\sum}_{p={0}}^{N}A^p)\T\, \ebi]_s = \rho^2_{ks}\rho_{si}= \rho_{ks}\,\rho_{ki}$ with $\max(i-N,k)\leq s\leq i$. Note that, the common nonzero positions of $\ebi\T({\sum}_{p={0}}^{N}A^p)$ and $\Wc^k({\sum}_{p={0}}^{N}A^p)\T\, \ebi$ are between $\max(i-N,k)$ and $i$. Therefore, $\ebi\T\, ({\sum}_{p={0}}^{N}A^p)\,\Wc^k\,({\sum}_{p={0}}^{N}A^p)\T\, \ebi = {\sum}_{s=\max(i-N,k)}^{i}\rho_{ks}\rho_{si}\rho_{ki}=(i+1-\max(i-N,k))\rho^2_{ki}$. The summation of the three terms establishes the required result.
\end{enumerate}

For the second result in {\emph{(ii)}}, recall
${\sum}_{t=1}^{\infty}(\tfrac{\pd \widetilde{F}_{ij}}{\pd u})^2={\sum}_{t=1}^{\infty}[(\tfrac{\pd x_i}{\pd u})^2+(\tfrac{\pd x_j}{\pd u})^2-2\tfrac{\pd x_i}{\pd u}\tfrac{\pd x_j}{\pd u}].$ Here $\Tc_1={\sum}_{t=1}^{\infty}(\tfrac{\pd x_i}{\pd u})^2,\;\Tc_2={\sum}_{t=1}^{\infty}(\tfrac{\pd x_j}{\pd u})^2$ for $j=i+1$ and are computed using Part 1.  The cross-term  forms $2\Tc_3+\Tc_4$ and can be computed as
$$2\,\tfrac{\pd x_i}{\pd u}\tfrac{\pd x_j}{\pd u}= 2\,\,\ebi\T \Qc^k \ebj.$$ The resulting $\ebi\T \Qc^k \ebj$ has four terms, and we compute them separately.
\begin{enumerate}[leftmargin=0.5cm, wide, labelwidth=!, labelindent=0pt, label=(\alph*)]
\item The first term $\ebi\T{\sum}_{t=1}^{N} \Pc^k_t\,\ebj = 0 $ because $\Pc^k_t$ is a diagonal matrix for all $t$.

\item The second term $\ebi\T{\sum}_{t=2}^{N}{\sum}_{p=1}^{t-1}\,(A^{t-p}\,\Pc_p^k)\,\ebj=0$ when $j=i+1$ by the same argument as in (ii).
 
\item The third term is  $\Tc_3=\ebi\T\,{\sum}_{t=2}^{N}\,{\sum}_{p=1}^{t-1}\,(\Pc_p^k {A^{t-p}}\T)\,\ebj$. This is $0$  for $i<k$. For $i\geq k$, $\ebi\T\,\Pc_p^k {A^{t-p}}\T\ebj\neq 0$ only when $p=t-1$ and $t\geq i-k+2$ and equals $\rho^2_{ki}\,\rho_{i,i+1}$. Therefore, for $N\geq i-k+2$, the expression $\ebi\T\,{\sum}_{t=2}^{N}\,{\sum}_{p=1}^{t-1} (\Pc_p^k {A^{t-p}}\T)\ebj = (N+k-i-1)\rho^2_{ki}\,\rho_{i,i+1}$, while for $N<i-k+2$, $\ebi\T\Pc_p^k = 0$ for all $p$. 
 
\item Let the fourth term be
\[\Tc_4=\ebi\T\,({\sum}_{p={0}}^{N}A^p)\,\Wc^k\,({\sum}_{p={0}}^{N}A^p)\T\, \ebj.\] For $j=i+1$,   the vector  $[({\sum}_{p={0}}^{N}A^p)\T\, \eb_{i+1}]_s = \rho_{s,i+1}$ is nonzero for $\max(i+1-N,1)\leq s\leq i+1$. Observe that, $\Wc^k\,({\sum}_{p={0}}^{N}A^p)\T\, \eb_{i+1}=0$ for the condition $i+1<k$.  Next, for $i+1\geq k$, we have $[\Wc^k({\sum}_{p={0}}^{N}A^p)\T\, \eb_{i+1}]_s = \rho^2_{ks}\rho_{s,i+1}= \rho_{ks}\rho_{k,i+1}$ with $\max(i+1-N,k)\leq s\leq i+1$. The nonzero elements of the row vector $[\ebi\T({\sum}_{p={0}}^{N}A^p)]_s = \rho_{si}$ with $\max(i-N,1)\leq s\leq i$.  Meanwhile, the common nonzero positions of $\ebi\T({\sum}_{p={0}}^{N}A^p)$ and $\Wc^k({\sum}_{p={0}}^{N}A^p)\T\, \eb_{i+1}$ are between $\max(i+1-N,k)$ and $i$. Therefore, $\ebi\T({\sum}_{p={0}}^{N}A^p)\,\Wc^k\,({\sum}_{p={0}}^{N}A^p)\T\, \eb_{i+1} = (i+1-\max(i-N+1,k))\rho^2_{ki}\rho_{i,i+1}.$ In particular, $\ebi\T ({\sum}_{p={0}}^{N}A^p)\,\Wc^k\,({\sum}_{p={0}}^{N}A^p)\T\, \eb_{i+1} =0$ for $i=k-1$.
\end{enumerate}
Combining $\Tc_1$, $\Tc_2$, $\Tc_3$, $\Tc_4$ completes the proof for impulse-train input.
\end{proof}

Theorem~\ref{thm-line-network-step-discrete} leverages closed-form expressions of the Gramian to provide a refinement of Theorem~\ref{thm-first-order-effects-step-discrete} for the special class of directed-line networks. For these networks, the first-order effects reveal an explicit dependence on edge weights, the duration of the input, and the distance between the node where the input is applied and the node where the effect is observed. Also, note that due to the directed nature of the network, certain nodes/edges remain unaffected by any disturbances affecting the system at a node later in the network.

\subsection{Continuous-Time Directed Line Network}
Now, the directed line networks following continuous-time dynamics as in \eqref{eq-dynamics-continuous} have self-loops, i.e., $a_{ii}= c,\; c<0 \,\forall\,\,i$. 
Given the lower triangular structure of the adjacency matrix,
the condition $c<0$ (with sufficiently large $\vert c \vert$) ensures the stability of the network dynamics. We first derive several properties of the directed line network with continuous-time dynamics.

\begin{lemma}\longthmtitle{Jordan normal form decomposition of adjacency matrix}
\label{lemma-jordan}
Consider the weighted adjacency matrix $A$ of the continuous-time directed line network. Then $AV=VJ$ where 
\begin{enumerate}[leftmargin=0.5cm,label=(\roman*)]
    \item $J_{ii} = c$ for $1\leq i \leq n$, $J_{i,i+1} = 1$ for $1\leq i \leq n-1$ and rest of the elements $0$s.
    \item $V_{n-i+1,i}=\rho_{1,n-i+1}$ for $1\leq i \leq n$ else $V_{ij}=0.$
     \item $\Vi_{n-i+1,i}=\rho^{-1}_{1i}$ for $1\leq i \leq n$ else $\Vi_{ij}=0.$
\end{enumerate}
\end{lemma}
\begin{proof}
    Follows from \cite[Chapter 4, Section 4.9]{CDM:2023}.
\end{proof}

Note that each $i^{\text{th}}$ column of $V$ or $\Vi$ has only one nonzero element at position $n-i+1$; thus, $V$ or $\Vi$ has $n$ nonzero elements, one in each column. Next, we compute the finite-time horizon controllability Gramian $\Wc^k$.
\begin{lemma}\longthmtitle{Computation of $\Wc^k$}\label{lemma-gramian}
Consider the weighted adjacency matrix $A$ of the directed line network following continuous-time dynamics. Let the input be applied at the $k^{\text{th}}$ node. Then the structure of the controllability Gramian $\Wc^k$ is as follows
\begin{enumerate}[leftmargin=0.5cm,label=(\roman*)]
\item for an impulse input at $t=0$ the infinite-time  controllability Gramian is
\begin{align*}
    \Wc^k_{fg} = \begin{cases}
        \frac{\rho_{1f}\,\rho_{1g}\,\eta !}{\rho^2_{1k}(f-k)!\,(g-k)!\, (-2c)^{\eta+1}}& \text{for}\; f\geq k \; \text{and}\; g\geq k,\\
        0 \qquad \hspace{3cm} &\text{otherwise}.
        \end{cases}
\end{align*}
        \item the finite-time controllability Gramian for $t\in[0,\tau]$ is
        \begin{align*}
    [\Pc^k_{\tau}]_{fg} = \begin{cases}
        \frac{\rho_{1f}\,\rho_{1g}\,\eta !}{\rho^2_{1k}(f-k)!\,(g-k)!\, (-2c)^{\eta+1}}&\hspace{-0.5cm}\big[1-e^{2c\tau}\sum_{p=0}^\eta \frac{(-2c\tau)^p}{p!}\big],\\
        \hspace{1cm}&\text{for}\,\, f\geq k \; \text{and}\; g\geq k,\\
        0 \quad \hspace{3cm} &\text{otherwise},
    \end{cases}
\end{align*}    
\end{enumerate}
where $\eta=f+g-2k$.
\end{lemma}

\begin{proof}
Let $\mathcal{T} = e^{At}\,\eb_k\,\eb_k\T\,e^{A\T t}$. By Theorem \ref{thm-first-order-effects-step-continuous}, $\Wc^k=\dspi_{0}^{\tau}\mathcal{T}\, \td t$. Applying Lemma \ref{lemma-jordan}, \cite[Chapter 4]{CDM:2023}, and $A= V J \Vi$, we get $$\mathcal{T} = V e^{Jt}\,\Vi\eb_k\,\eb_k\T\, \Vit e^{J\T t} V\T=[Ve^{Jt}\,\Vi\ebk]\,[Ve^{Jt}\,\Vi\ebk]\T.$$ The vector $\Vi\ebk$ has only one nonzero element at position $n-k+1$ with value $[\Vi\ebk]_{n-k+1}={1}/{\rho_{1k}}$. Since $e^{Jt}$ is an upper-triangular matrix, $[e^{Jt}]_{fg}=\frac{e^{ct}t^{g-f}}{(g-f)!}$ for $g\geq f$ and $[e^{Jt}]_{fg}=0$ otherwise. The vector $e^{Jt}\Vi\ebk$ is therefore 
    \begin{align*}
        [e^{Jt}\Vi\ebk]_s = \begin{cases}
            \frac{e^{ct}t^{n-k+1-s}}{\rho_{1k}(n-k+1-s)!}, \quad &\text{for} \quad 1\leq s \leq n-k+1,\\
            0, \hspace{2.3cm} &\text{for} \quad n-k+1 <s\leq n.
        \end{cases}
    \end{align*}
   On left-multiplying by $V$, $Ve^{Jt}\Vi\ebk$ is 
    \begin{align*}
        [Ve^{Jt}\Vi\ebk]_s = \begin{cases}
        0, \hspace{2.3cm} &\text{for} \quad 1 <s\leq k-1,\\
            \frac{\rho_{1s}\,e^{ct}t^{s-k}}{\rho_{1k}(s-k)!}, \quad\qquad &\text{for} \quad k\leq s \leq n.
                 \end{cases}
    \end{align*}
    Consequently, 
    \begin{align*}
    \mathcal{T}_{fg} = \begin{cases}
        \frac{\rho_{1f}\,\rho_{1g}\,e^{2ct}t^{\eta}}{\rho^2_{1k}(f-k)!\,(g-k)!},&\text{for}\quad f\geq k \; \text{and}\; g\geq k,\\
        0 \qquad \hspace{0cm} &\text{otherwise}.
    \end{cases}
\end{align*}

We obtain {\emph{(i)}} by integrating $\mathcal{T}_{fg}$ with respect to $t$ using \cite[3.351-3]{ISG-IMR:2000} with limits $t=0$ and $t\rightarrow \infty$. Similarly, we obtain \emph{(ii)} by integrating $\mathcal{T}_{fg}$ with respect to $t$ using \cite[3.351-1]{ISG-IMR:2000} with limits $t=0$ and $t=\tau$.
 \end{proof}
 
Note that $\Wc^k$ is a block-diagonal matrix with the first $k-1$ rows and columns being zero. 
Next, we compute $\Phi_q\,\ebi$.
\begin{lemma}\longthmtitle{Computation of $\Phi_q\,\ebi$}\label{lemma-vqi}
Consider the weighted adjacency matrix $A$ of the directed line network following continuous-time dynamics. For the unit vector $\ebi$, we have 
\begin{align*}
[\Phi_q\,\ebi]_s = \begin{cases}
        \frac{\rho_{1i} \mathlarger{\sum}_{p=0}^{q}e^{pc\tau} (p\tau)^{i-s}}{\rho_{1s}(i-s)!} \qquad & \text{for}\quad 1\leq s \leq i,\\
        0 \qquad  &\text{otherwise}.
    \end{cases}
\end{align*}
\end{lemma}
\begin{proof}
    From Theorem \ref{thm-first-order-effects-step-continuous} we have $$\Phi_q\,\ebi=\Big(\mathlarger{\sum}_{p=0}^{q}e^{pA\tau}\Big)\T\eb_i=\Vit \mathlarger{\sum}_{p=0}^{q}e^{pJ\T\tau}V\T\ebi.$$ The vector
    $V\T \ebi$ has only one nonzero element at position $n-i+1$ whose value is $[V\T \ebi]_{n-i+1}=\rho_{1i}$. As $e^{pJ\T\tau}$ is a lower-triangular matrix, we get $[e^{pJ\T \tau}]_{fg}=\frac{e^{pct}(p\tau)^{f-g}}{(f-g)!}$ for $f\geq g$ and $[e^{pJ\T \tau}]_{fg}=0$ otherwise. This leads to
    \begin{align*}
        &[e^{pJ\T \tau} V\T\ebi]_s = \begin{cases}
        0, \quad&\text{for} \quad 1\leq s <n+1-i,\\
            \frac{\rho_{1i}\,e^{pct}(p\tau)^{s-n+i-1}}{(s-n+i-1)!}, &\text{for} \quad n+1-i\leq s \leq n.
            \end{cases}\\
            &[\Vit e^{pJ\T \tau} V\T\ebi]_s = \begin{cases}
          \frac{\rho_{1i}\,e^{pct}(p\tau)^{i-s}}{\rho_{1s}(i-s)!}, &\text{for} \quad 1\leq s \leq i,\\
          0, & \text{for} \quad i< s\leq n.
            \end{cases}
    \end{align*}
Taking the sum from $p=0$ to $q$ gives the required result.
\end{proof}

Using Lemma \ref{lemma-vqi}, and defining the vector $\vqid$ as
\begin{align}
    \label{vqi-diff}
    &\vqid = \Phi_q\,\ebi- \Phi_q\,\eb_{i+1},\text{we get}\\
   \notag &[\vqid]_s = \begin{cases}
        \frac{\rho_{0,i-1}}{\rho_{0,s-1}(i-s)!} \mathlarger{\sum}_{p=0}^{q}e^{pc\tau} (p\tau)^{i-s}\bigg(1-\frac{a_{i+1,i}p\,\tau}{i+1-s}\bigg),&\\
        &\hspace{-3cm}\text{for}\; 1\leq s \leq i,\\
        -\sum_{p=0}^{q}e^{pc\tau} \;  &\hspace{-3cm}\text{for} \quad s=i+1,\\
        0 \;  &\hspace{-3cm}\text{otherwise}.
    \end{cases}
\end{align}
As we are equipped with the necessary preliminaries, we will proceed to present analytical expressions for the first-order effect of impulse input on directed line networks, for continuous-time dynamics, followed by the impulse-train input case. The following result is a generalization of \cite[Theorem 6]{PVC-BKP-JC:2023}, as it extends the results to all network edges. 

 \begin{theorem}\longthmtitle{Impulse input effect in directed line networks  with continuous-time dynamics}\label{thm-line-network-impulse-continuous}
Consider the continuous-time directed line network dynamics in \eqref{eq-dynamics-whole}, and let an impulse input $\ub = u\delta(t)$ with $u\in\real$ be applied at node $k$. Then  we have the following for the node $i$ and the edge $i\rightarrow j$ with $j=i+1$,
    \begin{enumerate}[leftmargin=0.5cm, label=(\roman*)]
\item $\dspi_0^{\infty}\Big(\dfrac{\pd x_i}{\pd u}\Big)^2 \textup{d} t= \begin{cases}
0, \quad &\text{for}\quad i<k,\\
\frac{\rho^2_{1i}\ (2i-2k) !}{\rho^2_{1k}((i-k)!)^2\, (-2c)^{2i-2k+1}}& \text{for}\quad  i> k.
\end{cases}$
\item $ \dspi_0^{\infty}\!\!\Big(\dfrac{\pd \widetilde{F}_{ij}}{\pd u}\Big)^2\textup{d} t=\begin{cases}
    0, \quad &\hspace{-3.5cm}\text{for}\quad i<k-1,\\
     -\dfrac{1}{2c}, \quad &\hspace{-3.5cm}\text{for}\quad i=k-1,\\   
    \Big[1+\dfrac{a_{i+1,i}(2i-2k+1)(2c+a_{i+1,i})}{2c^2(i-k+1)}\Big]\Wc^k_{ii}\\
    &\hspace{-3.5cm}\text{for}\quad i\geq k.
\end{cases} $
	 \end{enumerate}
where $\widetilde{F}_{ij}=F_{ij}/\alpij$.
    \end{theorem}
    
\begin{proof}
\begin{enumerate}[leftmargin=0.5cm, label=\emph{(\roman*)}]
\item This result is proved by invoking Theorem \ref{thm-first-order-effects-impulse} and setting $f=g=i$ in 1) of Lemma \ref{lemma-gramian}.
\item Using Lemma \ref{lemma-gramian} for $j=i+1$ we get
\begin{align*}
\Wc^k_{i,i+1} &= -\dfrac{a_{i+1,i}(2i-2k+1)}{2c(i-k+1)}\Wc_{ii}=-\frac{c}{a_{i+1,i}}\Wc^k_{i+1,i+1}.
\end{align*}
\end{enumerate} Now, using Theorem \ref{thm-first-order-effects-impulse}, we get the required result. 
\end{proof}

Theorem~\ref{thm-line-network-impulse-continuous}  refines Theorem~\ref{thm-first-order-effects-impulse} for the Gramian for the special case of directed-line networks. We note that the nodal distance of the observed node $i$ from the input node $k$ plays a significant role in determining the effect of the disturbance. Finally, we consider the impulse-train input case.
    
\begin{theorem}\longthmtitle{Impulse-train input effect in directed line networks  with continuous-time dynamics}\label{thm-line-network-impulse-train-continuous}
Consider the continuous-time directed line network dynamics in \eqref{eq-dynamics-whole}.  Let a train of impulse inputs of magnitude $u\in\real$ be applied at node $k$ at $t = t_0,\, t_1,\ldots,t_N$ such that $t_{q+1}-t_q = \tau$ $\forall$ $q = 0,\,1,\ldots,N.$   Then  we have the following for the node $i$ and the edge $i\rightarrow j$ with $j=i+1$,
\begin{enumerate}[leftmargin=0.5cm, label=(\roman*)]
\item $\dspi_0^{\infty}\Big(\dfrac{\pd x_i}{\pd u}\Big)^2 \textup{d} t= \begin{cases}
0, \qquad \text{for}\quad i<k,\\
\mathlarger{\sum}_{q=0}^{N-1}\,\mathlarger{\sum}_{f=k}^{i}\,\mathlarger{\sum}_{g=k}^{i}\bigg\{[\Phi_q\,\ebi]_f[\Phi_q\,\ebi]_g\,[\Pc^k_{\tau}]_{fg}\\ +[\Phi_N\,\ebi]_f[\Phi_N\,\ebi]_g\,\Wc^k_{fg}\bigg\},\, \, \text{for}\, \,i\geq k.
\end{cases}$
\item $ \dfrac{1}{\cij^2}\,\dspi_0^{\infty}\Big(\dfrac{\pd F_{ij}}{\pd u}\Big)^2\textup{d} t=\begin{cases}
0, \quad \text{for}\quad i<k-1,\\
\mathlarger{\sum}_{q=0}^{N-1}\,\mathlarger{\sum}_{f=k}^{i+1}\,\mathlarger{\sum}_{g=k}^{i+1}\bigg\{[\vqid]_f[\vqid]_g\,[\Pc^k_{\tau}]_{fg}\\ +[\mu^i_N]_f[\mu^i_N]_g\,\Wc^k_{fg}\bigg\},\;\text{for}\;i\geq k-1.
\end{cases}$
\end{enumerate}
\end{theorem}
\begin{proof}
   Using \eqref{vqi-diff}, Lemma \ref{lemma-vqi}, and Theorem \ref{thm-first-order-effects-step-continuous}, we get the required results. It is worth noting that in addition to the factors influencing vulnerability, from the previous theorems, the length of the input train also plays a role. 
\end{proof}

Similarly, Theorem \ref{thm-line-network-impulse-train-continuous} is a refinement of Theorem \ref{thm-first-order-effects-step-continuous}.
In contrast to the Gramian-based expressions of Section~\ref{nodal-input}, the results of this section for directed line networks bring to the fore the explicit role of edge weights on the network vulnerability, duration of disturbance inputs, and nodal distance between input and observed node, among other factors. This enables a deeper understanding of how the underlying network structure affects the system's vulnerability.

\section{Numerical Experiments}\label{examples}
Here we present numerical simulations to illustrate our results and demonstrate their efficacy and applicability.
We consider directed line and Erd\H{o}s-R\'{e}nyi networks. Our treatment covers both discrete-time and continuous-time dynamics, further categorized by impulse and step-input scenarios.

\subsubsection*{Motivating Scenario}
To motivate the numerical experiments, consider the following attacker-defender framework. An attacker may inject impulse or step disturbances at selected nodes to maximize transient flow deviations across network edges. The defender prepares for the attack by preemptively identifying vulnerabilities and taking appropriate remediation.

Such a framework has direct applications to real-world critical infrastructure networks, such as power grids, communication and cyber-physical networks, and transportation networks. In power grids, such attacks could include false data injection, load-altering attacks, or physical impulses (either generator or load trips). In communication networks, these could include targeted DDoS pulses and malware injections that spike link loads. Here, the defender employs the VM to identify which nodes will need strengthening to reduce attack effectiveness. The procedure to do so is described in Algorithm~\ref{alg:validation_a}.

\begin{algorithm}[htb]
\caption{Vulnerability Matrix-based Nodal Influence}
\label{alg:validation_a}
\begin{algorithmic}[1]
    \State Compute the influence of each node in the network from the Vulnerability Matrix as defined in Section \ref{nodal-input}
    \State Sort the nodes in descending order of nodal influence
    \State Form a set $\Vc_s$ of the top $N_s$ most influential nodes from the sorted nodes
\end{algorithmic}
\end{algorithm}

Algorithm~\ref{alg:validation_b} employs these performance metrics to validate our VM-based approach (cf. Algorithm~\ref{alg:validation_a}) against actual time-domain dynamic simulations, with an effectiveness metric $\kappa$. A higher $\kappa$ indicates a better identification of vulnerable nodes.

\begin{algorithm}[htb]
\caption{Validation of identification of vulnerable nodes}
\label{alg:validation_b}
\begin{algorithmic}[1]
    \State Simulate the actual dynamics of the network and compute the flow profiles with inputs (for both step and impulse) at individual nodes
    \State For each nodal input, compute the effect through $\Jc_2$ or $\Jc_{\infty}$ performance by simulating the network dynamics for the time duration $T$
    \State Rank the nodes in descending order according to the selected performance metric and form the set $\Vc_f$ by selecting the top $N_s$ nodes
    \State Measure the effectiveness by evaluating $\kappa = \tfrac{|\Vc_s \cap \Vc_f|}{N_s}$, where $\Vc_s$ is determined by Algorithm~\ref{alg:validation_a}
\end{algorithmic}
\end{algorithm}

\subsection{$7$-Node Directed Line Network}\label{example-7-node}
Consider a  $7-$node network with the following edge weights: $a_{21}=0.7,\; a_{32}=0.8,\;a_{43}=0.9,\; a_{54}=0.6,\;a_{65}=0.7,\;a_{76}=0.5$. When observing discrete-time dynamics, we assume no self-loops, whereas we consider self-loops with weight $-1$ when the system follows continuous-time dynamics. For this network, we note that the number of non-self-loop edges in the network is $n_e=6$ and the set $\Ec_A = \{\,(1,2),\,(2,3),\,(3,4),\, (4,5),\, (5,6),\, (6,7)\,\}$. Next, consider nodal impulse/step inputs of magnitude $u=50$, with a step input duration of $T_s = 5$ and a simulation time of $T=30$. The initial states are assumed to be $0$. 

\textit{(i) Discrete-time dynamics:} For the discrete-time dynamics,  we use Theorem~\ref{thm-first-order-effects-impulse} to compute the VM matrix, with the results compiled in Table \ref{table:VM-discrete-impulse} (we use shades of red color for impulse inputs and shades of blue color for step inputs).

\begin{table}[htbp] 
\centering
\newcolumntype{C}{>{\centering\arraybackslash}X}
\addtolength{\tabcolsep}{-0.1em}
\begin{tabular}{lS[table-format=2.2]*{8}{S}}
\toprule 
\addlinespace
{$k$}& $\VM_{12}$ & $\VM_{23}$ & $\VM_{34}$ & $\VM_{45}$& $\VM_{56}$ & $\VM_{67}$ & $\nu_k$
\tabularnewline
\cmidrule[\lightrulewidth](lr){1-8}\addlinespace[1ex]
$1$ & 0.73 & 0.51 & 0.46 & 0.12 & 0.07 & 0.01 & \cellcolor[HTML]{FF2E2E}1.91 \tabularnewline
$2$ & 0.49 & 1.05 & 0.94 & 0.25 & 0.14 & 0.03 & \cellcolor[HTML]{D10000}2.90  \tabularnewline
$3$ & 0 & 0.64 & 0.47 & 0.40 & 0.21 & 0.05 & \cellcolor[HTML]{FF0000}2.76 \tabularnewline
$4$ & 0 & 0 & 0.81 & 0.49 & 0.26 & 0.06 & \cellcolor[HTML]{FF5C5C}1.62 \tabularnewline
$5$ & 0 & 0 & 0 & 0.36 & 0.73 & 0.15 & \cellcolor[HTML]{FF8A8A}1.24 \tabularnewline
$6$ & 0 & 0 & 0 & 0 & 0.49 & 0.31 & \cellcolor[HTML]{FFB8B8}0.80 \tabularnewline
$7$ & 0 & 0 & 0 & 0 & 0 & 0.25 & \cellcolor[HTML]{FFE6E6}0.25\tabularnewline
\addlinespace
\bottomrule
\end{tabular}
\caption{The Vulnerability Matrix for the discrete-time case with impulse input, with shaded vulnerability influence $\nu_k$-- darker shade indicates a greater potential target.}
\label{table:VM-discrete-impulse}
\smallskip
\end{table}

Next, we numerically simulate the dynamics in the time interval $ [0,\, T]$  and compile the $\Jc_{\infty}^f$ and  $\Jc_2^f$ edge flow-related performance metrics in Table \ref{table:max-flow-discrete-impulse}. Note that the largest $\Jc_{\infty}^f$ and  $\Jc_2^f$ flows occur when the input is applied at node $2$, which is also the node with the maximum nodal influence. Consequently, node $2$ of the network serves as the optimal location for an attacker to launch an impulse-input attack.

\begin{table}[htbp] \centering
\newcolumntype{C}{>{\centering\arraybackslash}X}
\addtolength{\tabcolsep}{-0.27em}
\begin{tabular}{lS[table-format=2.2]*{6}{S}| *{1}{S}}
\toprule 
\addlinespace
{$k$}& {$\mathcal{F}_{12}$} & $\mathcal{F}_{23}$ & $\mathcal{F}_{34}$ & $\mathcal{F}_{45}$& $\mathcal{F}_{56}$ & $\mathcal{F}_{67}$ & $\Jc_{\infty}^f$ & {$\Jc_{2}^f\times10^3$}
\tabularnewline
\cmidrule[\lightrulewidth](lr){1-9}\addlinespace[1ex]
$1$ &35.0& 28.0& 25.2 & 15.1 & 10.6 &5.3 & \cellcolor[HTML]{FF2E2E}119.2 & \cellcolor[HTML]{FF2E2E}4.8 \tabularnewline
$2$ & 35.0 & 40.0 & 36.0 & 21.6 & 15.1 & 7.6 & \cellcolor[HTML]{D10000}155.3 &\cellcolor[HTML]{D10000}7.2 \tabularnewline
$3$ & 0 & 40.0 & 45.0 & 27.0 & 18.9 & 9.5 & \cellcolor[HTML]{FF0000}140.4 &\cellcolor[HTML]{FF0000}6.9\tabularnewline
$4$ & 0 & 0 & 45.0 & 30.0 & 21.0 & 10.5 & \cellcolor[HTML]{FF5C5C}106.5 &\cellcolor[HTML]{FF5C5C}4.0\tabularnewline
$5$ & 0 & 0 & 0 & 30.0 & 35.0 & 17.5& \cellcolor[HTML]{FF8A8A}82.5 &\cellcolor[HTML]{FF8A8A}3.1 \tabularnewline
$6$ & 0 & 0 & 0 & 0 & 35.0 & 25.0 & \cellcolor[HTML]{FFB8B8}60.0 &\cellcolor[HTML]{FFB8B8}2.0\tabularnewline
$7$ & 0 & 0 & 0 & 0 & 0 & 25.0 & \cellcolor[HTML]{FFE6E6}25.0 &\cellcolor[HTML]{FFE6E6}0.6\tabularnewline
\addlinespace
\bottomrule
\end{tabular}
\caption{The maximum flow deviations (where $\mathcal{F}_{ij}=\max|F_{ij}|$) on the edges for the discrete-time case with impulse inputs, along with the $\Hc_2$ and $\Hc_{\infty}$ performance metrics-- a darker shade indicates a greater potential target.}
\label{table:max-flow-discrete-impulse}
\end{table}

For a step input, the VM is computed using Theorem \ref{thm-line-network-step-discrete} and is presented in Table \ref{table:VM-discrete-step}. Note that here $N=T_s$. When numerically simulating the flows based on system dynamics, we observe that for the step input, the metric $\Jc_{\infty}^f$ and the max flow deviations are identical to the impulse input scenario as tabulated in Table \ref{table:max-flow-discrete-impulse}. The values of $\Jc_2^f$ for the step-input case are tabulated in Table \ref{table:VM-discrete-step}. We note that, similar to the impulse input scenario, there is alignment in the most vulnerable node when comparing the two criteria: vulnerability influence and flow-based performance metrics.

\begin{table}[h] \centering
\newcolumntype{C}{>{\centering\arraybackslash}X}
\addtolength{\tabcolsep}{-0.27em}
\begin{tabular}{lS[table-format=2.2]*{6}{S}| *{1}{S}}
\toprule 
\addlinespace
{$k$}& $\VM_{12}$ & $\VM_{23}$ & $\VM_{34}$ & $\VM_{45}$& $\VM_{56}$ & $\VM_{67}$ & $\nu_k$ & ${\Jc_2^f\times 10^4}$
\tabularnewline
\cmidrule[\lightrulewidth](lr){1-9}\addlinespace[1ex]
$1$ & 0.95& 0.58 & 0.47 & 0.20 & 0.09 & 0.03 &\cellcolor[HTML]{B8DCFF}  2.32 & \cellcolor[HTML]{B8DCFF} 0.58 \tabularnewline
$2$ & 2.94 & 1.18 & 0.96 & 0.40 & 0.18 & 0.06& \cellcolor[HTML]{2E9AFF}  5.72 &\cellcolor[HTML]{2E9AFF} 1.43\tabularnewline
$3$ & 0 & 3.84 & 1.51 & 0.63 & 0.28 & 0.09 & \cellcolor[HTML]{006CD1} 6.35 &\cellcolor[HTML]{006CD1} 1.58\tabularnewline
$4$ & 0 & 0 & 4.86 & 0.78 & 0.34& 0.11& \cellcolor[HTML] {0084FF} 6.09& \cellcolor[HTML]{0084FF} 1.52\tabularnewline
$5$ & 0 & 0 & 0 & 2.16 &0.95 & 0.31 &\cellcolor[HTML]{8AC6FF} 3.42 &\cellcolor[HTML]{8AC6FF} 0.85\tabularnewline
$6$ & 0 & 0 & 0 & 0 & 2.94 & 0.63 & \cellcolor[HTML]{5CB0FF} 3.57 & \cellcolor[HTML]{5CB0FF} 0.89\tabularnewline
$7$ & 0 & 0 & 0 & 0 & 0 & 1.50 &\cellcolor[HTML]{E6F3FF} 1.50& \cellcolor[HTML]{E6F3FF} 0.38\tabularnewline
\addlinespace
\bottomrule
\end{tabular}
	\caption{The Vulnerability Matrix for the discrete-time case with step input with shaded vulnerability influence $\nu_k$ and the corresponding $\Hc_2$ performance.}
	\label{table:VM-discrete-step}
\end{table}

\textit{(ii) Continuous-time dynamics:} 
We use MATLAB `\textit{ode45}' with a time span $[0,\, T]$ for simulating the continuous-time dynamics. Here, for the step-input case, we take $N=10$.

Next, we follow a simulation approach identical to the one used in the discrete-time case. The VM is computed using Theorem~\ref{thm-first-order-effects-impulse} and compiled in Table~\ref {table:VM-continuous-impulse}. 
\begin{table}[htbp] 
\centering
\newcolumntype{C}{>{\centering\arraybackslash}X}
\addtolength{\tabcolsep}{-0.2em}
\begin{tabular}{lS[table-format=2.2]*{7}{S}}
\toprule 
\addlinespace
{$k$}& $\VM_{12}$ & $\VM_{23}$ & $\VM_{34}$ & $\VM_{45}$& $\VM_{56}$ & $\VM_{67}$ & $\nu_k$
\tabularnewline
\cmidrule[\lightrulewidth](lr){1-8}\addlinespace[1ex]
$1$ & 0.13 & 0.02 & 0.01 & 0 & 0 & 0 &\cellcolor[HTML]{FFB8B8}0.16\tabularnewline
$2$ & 0.25& 0.17 & 0.03 & 0.01 & 0& 0 & \cellcolor[HTML]{FF2E2E} 0.46\tabularnewline
$3$ & 0 & 0.32 & 0.20 & 0.03 & 0 & 0& \cellcolor[HTML]{D10000} 0.55\tabularnewline
$4$ & 0 & 0 & 0.41 & 0.10 & 0.01 & 0 &\cellcolor[HTML]{FF0000} 0.52 \tabularnewline
$5$ & 0 & 0 & 0 & 0.18 & 0.13 & 0.01 & \cellcolor[HTML]{FF8A8A} 0.32 \tabularnewline
$6$ & 0 & 0 & 0 & 0 & 0.25 & 0.08 & \cellcolor[HTML]{FF5C5C} 0.33 \tabularnewline
$7$ & 0 & 0 & 0 & 0 & 0 & 0.13 & \cellcolor[HTML]{FFE6E6} 0.13\tabularnewline
\addlinespace
\bottomrule
\end{tabular}
	\caption{The Vulnerability Matrix for the continuous-time case with impulse input, with shaded vulnerability influence $\nu_k$-- darker shade indicates a greater potential target.}
	\label{table:VM-continuous-impulse}
\end{table}
The $\Jc_{\infty}^f$ and $\Jc_2^f$ performance metrics are tabulated in Table \ref{table:max-flow-continuous-impulse}. Based on the VM, node $3$ emerges as the most influential and vulnerable node—a finding directly corroborated by both the $\Jc_{\infty}^f$ and $\Jc_2^f$ metrics. Consequently, node $3$ represents the most strategic target for a step-input attack.

Finally, we consider step inputs. Table~\ref{table:VM-continuous-step} presents the VM computed using Theorem~\ref{thm-line-network-impulse-train-continuous} and Table~\ref{table:max-flow-continuous-step} the flow-based performance metrics. We note that, similar to the discrete-time dynamics, while there is no consensus node, node $4$ is identified as the most vulnerable by at least two metrics and thus the most critical in the network.

\begin{table}[htbp] \centering
\newcolumntype{C}{>{\centering\arraybackslash}X}
\addtolength{\tabcolsep}{-0.27em}
\begin{tabular}{l S[table-format=2.2] *{6}{S} | *{1}{S}}
\toprule 
\addlinespace
{$k$}& $\mathcal{F}_{12}$ & $\mathcal{F}_{23}$ & $\mathcal{F}_{34}$ & $\mathcal{F}_{45}$& $\mathcal{F}_{56}$ & $\mathcal{F}_{67}$ & $\Jc_{\infty}^f$ & $\Jc_{2}^f$
\tabularnewline
\cmidrule[\lightrulewidth](lr){1-9}\addlinespace[1ex]
$1$ &0.34& 0.07 & 0.02& 0.01 & 0 & 0 & \cellcolor[HTML]{FFB8B8} 0.47 & \cellcolor[HTML]{FFB8B8} 0.04\tabularnewline
$2$ & 0.34 & 0.39 & 0.09 & 0.04 & 0.02 & 0 &\cellcolor[HTML]{FF0000} 0.88 &\cellcolor[HTML]{FF2E2E} 0.11\tabularnewline
$3$ & 0 & 0.39& 0.44& 0.07 & 0.03 & 0.01 & \cellcolor[HTML]{D10000} 0.94& \cellcolor[HTML]{D10000} 0.14\tabularnewline
$4$ & 0 & 0 & 0.44 & 0.29 & 0.05 & 0.02& \cellcolor[HTML]{FF2E2E} 0.80 & \cellcolor[HTML]{FF0000} 0.13\tabularnewline
$5$ & 0 & 0 & 0 & 0.29& 0.34& 0.05& \cellcolor[HTML]{FF5C5C} 0.68 & \cellcolor[HTML]{FF8A8A} 0.08\tabularnewline
$6$ & 0 & 0 & 0 & 0 & 0.34 & 0.25 & \cellcolor[HTML]{FF8A8A} 0.59 &\cellcolor[HTML]{FF8A8A} 0.08\tabularnewline
$7$ & 0 & 0 & 0 & 0 & 0 & 0.25 & \cellcolor[HTML]{FFE6E6} 0.25 & \cellcolor[HTML]{FFE6E6} 0.03\tabularnewline
\addlinespace
\bottomrule
\end{tabular}
\caption{The max flow deviations on the edges for the continuous-time case with impulse inputs, along with the $\Hc_2$ and $\Hc_{\infty}$ performance metrics-- darker shade indicates a greater potential target.}
\label{table:max-flow-continuous-impulse}
\end{table}

\begin{table}[h] \centering
\newcolumntype{C}{>{\centering\arraybackslash}X}
\addtolength{\tabcolsep}{-0.3em}
\begin{tabular}{lS[table-format=3.1]*{7}{S}}
\toprule 
\addlinespace
{$k$}& $\VM_{12}$ & $\VM_{23}$ & $\VM_{34}$ & $\VM_{45}$& $\VM_{56}$ & $\VM_{67}$ & $\nu_k$
\tabularnewline
\cmidrule[\lightrulewidth](lr){1-8}\addlinespace[1ex]
$1$ & 1.90& 0.63& 0.31& 0.27& 0.08& 0.04 & \cellcolor[HTML]{E6F3FF} 3.23 \tabularnewline
$2$ & 9.09 & 1.98& 0.81 & 0.61& 0.19 & 0.09
&\cellcolor[HTML]{2E9AFF} 12.76\tabularnewline
$3$ & 0 & 11.88&2.11& 1.10&0.33& 0.15& \cellcolor[HTML]{0084FF} 15.57\tabularnewline
$4$ & 0 & 0 & 15.03 & 1.81&6.48& 0.20 & \cellcolor[HTML]{006CD1} 17.52\tabularnewline
$5$ & 0 & 0 & 0 & 6.68& 1.91 & 0.62& \cellcolor[HTML]{8AC6FF} 9.21\tabularnewline
$6$ & 0 & 0 & 0 & 0 & 9.09 & 1.62& \cellcolor[HTML]{5CB0FF} 10.71\tabularnewline
$7$ & 0 & 0 & 0 & 0 & 0 &4.64& \cellcolor[HTML]{B8DCFF} 4.64\tabularnewline
\addlinespace
\bottomrule
\end{tabular}
	\caption{The Vulnerability Matrix for the continuous-time case with step input with shaded vulnerability influence $\nu_k$.}
 \vspace{-1mm}
	\label{table:VM-continuous-step}
\end{table}

\begin{table}[h] \centering
\newcolumntype{C}{>{\centering\arraybackslash}X}
\addtolength{\tabcolsep}{-0.27em}
\begin{tabular}{l S[table-format=2.2] *{6}{S} | *{1}{S}}
\toprule 
\addlinespace
{$k$}& $\mathcal{F}_{12}$ & $\mathcal{F}_{23}$ & $\mathcal{F}_{34}$ & $\mathcal{F}_{45}$& $\mathcal{F}_{56}$ & $\mathcal{F}_{67}$ & $\Jc_{\infty}^f$ & {$\Jc_{2}^f\times10^3$}
\tabularnewline
\cmidrule[\lightrulewidth](lr){1-9}\addlinespace[1ex]
$1$ &9.8& 5.8 & 4.0& 3.8 & 2.0 & 1.3 & \cellcolor[HTML]{B8DCFF} 26.8 &\cellcolor[HTML]{E6F3FF} 0.6\tabularnewline
$2$ & 21.2 & 10.3 & 6.6 & 5.6 & 3.1 & 2.0 & \cellcolor[HTML]{0084FF} 48.9& \cellcolor[HTML]{2E9AFF} 2.4 \tabularnewline
$3$ & 0 & 24.1& 10.7 & 7.4& 4.2 & 2.7& \cellcolor[HTML]{006CD1} 49.1& \cellcolor[HTML]{0084FF} 3.0\tabularnewline
$4$ & 0 & 0 & 27.3 & 9.2 & 5.0 & 3.2 & \cellcolor[HTML]{2E9AFF} 44.8& \cellcolor[HTML]{006CD1} 3.4 \tabularnewline
$5$ & 0 & 0 & 0 & 17.9& 9.8& 5.5& \cellcolor[HTML]{5CB0FF} 33.2& \cellcolor[HTML]{8AC6FF} 1.8\tabularnewline
$6$ & 0 & 0 & 0 & 0 & 21.3 & 8.5 & \cellcolor[HTML]{8AC6FF} 29.8 & \cellcolor[HTML]{5CB0FF} 2.1\tabularnewline
$7$ & 0 & 0 & 0 & 0 & 0 & 15.3 & \cellcolor[HTML]{E6F3FF} 15.3& \cellcolor[HTML]{B8DCFF} 0.9 \tabularnewline
\addlinespace
\bottomrule
\end{tabular}
\caption{The max flow deviations (where $\mathcal{F}_{ij}=\max|F_{ij}|$) on the edges for the continuous-time case with step inputs, along with the $\Hc_2$ and $\Hc_{\infty}$ performance metrics.}
\label{table:max-flow-continuous-step}
\end{table}

\subsection{Random Erd\H{o}s-R\'{e}nyi Networks}\label{example-erdos-renyi}
Next, we consider random Erd\H{o}s-R\'{e}nyi (ER) networks~\cite{PE-AR:60} to illustrate the efficacy of our proposed approach and its utility. To this end, we consider $1000$ random ER networks, each having $n=100$ nodes. As before, we study both discrete-time and continuous-time cases. For each network, the edge locations are selected with a probability of $0.3$. For the discrete-time case, the $A$ matrix is Schur stable~\cite{KZ-JD-KG:95} and without self-loops. For the continuous-time case, $A$ is rendered Hurwitz by the addition of negative weighted self-loops~\cite{KZ-JD-KG:95}. For these simulations, a step duration of $T_s=10$ is considered. 

We use Algorithm~\ref{alg:validation_a} to compute the VM and Algorithm~\ref{alg:validation_b} to evaluate its effectiveness in identifying the most vulnerable nodes. The VM computation relies on Theorems \ref{thm-first-order-effects-impulse}, \ref{thm-first-order-effects-step-discrete}, and \ref{thm-first-order-effects-step-continuous}. We note that $N=T_s$ in Theorem \ref{thm-first-order-effects-step-discrete} for the discrete-time case, and in Theorem \ref{thm-first-order-effects-step-continuous} for the continuous-time case; we set $N=20$ and $t_N=T_s$. For each of these cases, we form the set $V_s$, i.e., the set of sorted top $N_s = 0.1n$ influential input nodes computed from the VM. Subsequently, we numerically simulate the dynamics for each network for a time span of $[0,\, T]$ with $T=50$ and for the input magnitude $u=50$. The flow performances $\Jc_2^f$ and $\Jc^f_{\infty}$ are evaluated for each input node, and the set $\Vc_f$ is constructed for discrete-time dynamics (impulse and step inputs) and continuous-time dynamics  (impulse and step inputs). Finally, we determine the effectiveness measure $\kappa$ for each case. The results presented in Fig.~\ref{fig:performance} showcase the average $\kappa$ measure for the considered $1000$ random ER network sets, with a tight clustering around the mean. We also note that across the scenarios, there is a high degree of consistency, i.e., the nodal vulnerability as determined from the Vulnerability Matrix effectively captures the nodal influence evaluated through the flow performance metrics. The effectiveness measure is stronger for discrete-time dynamics and for the average flow metric $\Jc_2^f$ as compared to continuous-time and the max flow metric $\Jc_\infty^f$.

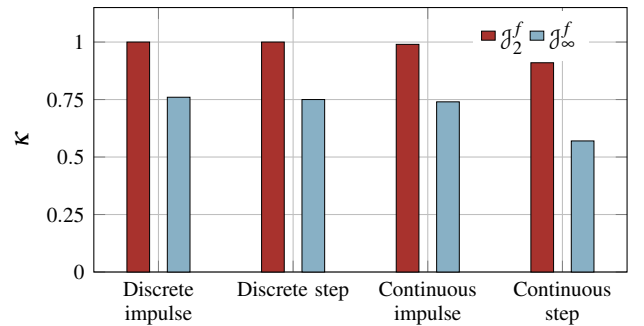
\begin{figure}[htbp]
\centering
\tikzset{every picture/.style={scale=1.0}}
\input{kappa.tex}
\caption{Average effectiveness measure $\kappa$ for $1000$ random Erd\H{o}s-R\'{e}nyi networks. The variance around the mean in each case for either $\Jc_2^f$ or $\Jc_\infty^f$ is never larger than $1.5e-2$.} 
\label{fig:performance}
\end{figure} 

\section{Conclusions}
We have presented a framework to analyze the effect of nodal inputs on edge flows in networks with stable linear (either discrete or continuous-time) dynamics. For both impulse and step inputs, we have provided explicit expressions for the first-order effects of the nodal inputs on system states and edge flows in terms of the corresponding controllability Gramians for general networks. We have particularized these results to the class of directed line networks, providing closed-form expressions in terms of edge weights, input duration, and distance in the graph to the nodal input. These results have led us to introduce the notion of a (relative) vulnerability matrix as a novel metric to help identify critical nodes in the network that could be potential targets for adversaries. Our numerical simulations demonstrate the tight correspondence between the topmost set of critical nodes identified from edge-flow performance indicators such as $\Hc_\infty$ and $\Hc_2$ and from the Gramian-informed vulnerability matrix, thus illustrating its efficacy. We believe the computational tractability of the metrics proposed here, along with the analytical insights that can be obtained from the explicit closed-form expressions, opens up the way to efficiently identify and exploit dependencies and strengthen large-scale vulnerable networks such as power grids and transportation networks. Future work will extend the analytical characterization of the vulnerability matrix to second-order models and develop game-theoretic strategies to counter adversarial nodal attacks.

\bibliographystyle{IEEEtran}
\bibliography{ref}
\end{document}

%% file: kappa.tex
%
%
\definecolor{mycolor1a}{RGB}{136,176,197}
\definecolor{mycolor2a}{RGB}{168,50,45}
\begin{tikzpicture}
\begin{axis}[
ylabel={$\kappa$},
ymin=0,
ymax=1.15,
yticklabel style = {font=\footnotesize,xshift=0ex},
xticklabel style = {font=\footnotesize,yshift=0ex},
axis background/.style={fill=white},
xmajorgrids,
ymajorgrids,
legend style={legend cell align=left, align=left, draw=black},
width=3.4in,
height=2in,
bar width=0.3cm,
xtick={1,2,3,4},
ytick={0, 0.25, 0.50, 0.75, 1.00},
xticklabel style={align=center, text width=1.5cm},
xticklabels={Discrete impulse, Discrete step, Continuous impulse, Continuous step},
every axis plot/.append style={ultra thin},
axis background/.style={fill=white},
legend style={at={(0.93,0.99)}, legend cell align=left, align=left, draw=black, font=\small, draw=none, legend columns=-1}
]

\addplot[ybar, fill=mycolor2a, draw=black, area legend] coordinates {
    (0.85, 1)
    (1.85, 1)
    (2.85, 0.99)
    (3.85, 0.91)
};

\addplot[ybar, fill=mycolor1a, draw=black, area legend] coordinates {
    (1.15, 0.76)
    (2.15, 0.75)
    (3.15, 0.74)
    (4.15, 0.57)
};

\legend{$\Jc_2^f$, $\Jc^f_{\infty}$ }
\end{axis}
\end{tikzpicture}